%% file: main.tex
\documentclass[onefignum,onetabnum]{siamart171218}
\usepackage[margin=4cm]{geometry}
\usepackage{enumitem}
\input{ex_shared}

\ifpdf
\hypersetup{
  pdftitle={A Viscosity Approach to Stochastic Differential Games of Control and Stopping Involving Impulsive Control},
  pdfauthor={David Mguni}
}
\fi

\myexternaldocument{ex_supplement}

\begin{document}

\maketitle

% REQUIRED
\begin{abstract}
  
This paper analyses a stochastic differential game of control and stopping in which one of the players modifies a diffusion process using impulse controls, an adversary then chooses a stopping time to end the game. The paper firstly establishes the regularity and boundedness of the upper and lower value functions from which an appropriate variant of the dynamic programming principle (DPP) is derived. It is then proven that the upper and lower value functions coincide so that the game admits a value and that the value of the game is a unique viscosity solution to a HJBI equation described by a double obstacle quasi-integro-variational inequality.
\end{abstract}

% REQUIRED

% REQUIRED
% \begin{AMS}
%   68Q25, 68R10, 68U05
% \end{AMS}

\section{Introduction}
This paper considers a stochastic differential game in which a controller modifies a diffusion process and an adversary stops the game. The objective of the controller is to exercise an impulse control that minimises an objective criterion whilst the adversary seeks to stop the process at a time which maximises the same function. A motivation for this investigation is a formal mathematical treatment of the probability of lifetime ruin problem which also includes financial transaction costs a problem in which an investor seeks to maximise their lifetime wealth by modifying their investment position whilst minimising the risk that their wealth process falls below some prefixed value.

Problems that combine discretionary stopping and stochastic optimal control have attracted much attention over recent years; in particular there is a significant amount of literature on models of this kind in which a single controller uses continuous controls to modify the system dynamics (see for example \cite{bayraktar2011proving,bayraktar2011minimizing,bayraktar2015minimizing}). Game-theoretic formulations of the problem in which the task of controlling the system dynamics and exit time is divided between two players who act according to separated interests have also been studied \cite{baghery2013optimal}. 

Impulse control problems are stochastic control models in which the cost of control is bounded below by some fixed positive constant which prohibits continuous control, therefore augmenting the problem to one of finding both an optimal sequence of times to apply the control policy, in addition to determining optimal control magnitudes. We refer the reader to \cite{bensoussan1982controle} as a general reference to impulse control theory and to \cite{vath2007model,palczewski2010finite} for articles on applications. Impulse control frameworks therefore underpin the description of financial environments with transaction costs and liquidity risks and more generally, applications of optimal control theory in which the system dynamics are modified by a sequence of discrete actions - see \cite{korn1999some} for an extensive survey of applications of impulse control models within finance. 
 	
Stochastic differential games with impulse control have also recently appeared in the impulse control literature. In \cite{zhang2011stochastic} stochastic differential games in which one player uses impulse control and the other uses continuous controls were studied. Using a verification argument, the conditions under which the value of game is a solution to a HJBI equation is also shown in \cite{zhang2011stochastic}. In \cite{cosso2013stochastic}, Cosso was the first to study a stochastic differential game in which both players use impulse control using viscosity theory. Thus, in \cite{cosso2013stochastic} it is shown that the game admits a value which is a unique viscosity solution to a double obstacle quasi-variational inequality. 

Despite the small but notable literature on impulse controls involving two players, the task of analysing stochastic differential games with a discretionary stopper and a controller who uses impulse controls remains unaddressed. 

Stochastic differential games of continuous control and stopping have a number of applications within theoretical finance. A notable example is the investor lifetime ruin problem in which an investor, who operates in a Black-Scholes market seeks to maximise some utility criterion which is a function of his wealth process whilst seeking some optimal time to exit the market which minimises the risk of ruin (see for example \cite{bayraktar2011minimizing} and \cite{bayraktar2015minimizing}). 

The restriction to continuous controls prohibits the inclusion of transaction costs for which the investor must take actions in a timed and discretised fashion therefore necessitating that the class of controls used to modify the system dynamics are impulse controls. The current paper therefore establishes the necessary results to construct optimal investment policies for the lifetime ruin problem with minimally bounded or transaction costs.

In control theory and more generally, differential game theory there are two main approaches to obtaining a solution to the problem. The first approach is a verification method which involves characterising the value function in terms of a set of (in general, non-linear, second order) PDEs or HJB equations (in the case of differential games, HJBI equations). The verification approach is a direct method that is initialised at the dynamic programming principle from which a classical limit procedure is used to then derive the HJB equations. 

Though the verification approach offers a direct means of establishing a PDE characterisation of the value function, verification theorems impose strong smoothness conditions on the value function that do not hold in a number of optimal control settings. Indeed, verification theorems require that the value function must be everywhere differentiable and have smooth derivatives almost everywhere.\footnote{The assumption that the value function is differentiable everywhere is sometimes referred to as the \textit{“high contact”} or \textit{“smooth fit”} principle.} In problems in which the controller makes discrete modifications to a diffusion process such smoothness assumptions are likely to be violated. In such cases, it is not possible to invoke the mean value theorem to derive the HJB(I) equations via a classical limiting procedure. A second issue is that verification theorems do not address the question of existence of the value of stochastic differential games nor is the question of uniqueness considered: in general, there exist an infinite number of Lipschitz continuous functions that satisfy the HJB(I) equations of verification theorems (see for example exercise 3.2 in \cite{cardaliaguet2010introduction}).

Fortunately, using viscosity theory it can be shown that the value functions of a wide class of stochastic control problems (and consequently stochastic differential games) do in fact satisfy verification theorems when the HJBI equations are interpreted in a weaker, viscosity sense. Indeed, viscosity solutions generalise the notion of a solution to a PDE to a non-classical definition\footnote{Viscosity solutions were introduced by Michael Crandall and Pierre-Louis Lions in 1983  \cite{crandall1983viscosity} and was developed to handle first order HJB equations. The theory was subsequently developed to handle second order equations in part due to a comparison principle result introduced by Robert Jensen in 1988 \cite{jensen1988maximum}.}.  The main advantage of the viscosity solution approach is that it does not require that the (everywhere) smoothness of the value function be established.
\section{Contribution}

This paper analyses a stochastic differential game of control and stopping in which the controller uses impulse controls; the results cover a general setting in which the underlying state process is a diffusion process. We examine the problem using viscosity theory from which we show that the value of the game exists and is a unique viscosity solution to a HJBI equation. A similar game to the one considered in this paper is that contained within \cite{cosso2013stochastic}. However, unlike in our setting in one of the players can choose to terminate the game,  the analysis in \cite{cosso2013stochastic} considers a stochastic differential game in which both players modify the state process using impulse controls and thus incorporates a symmetry in the class of controls used by the players. 
Other related papers to the current are \cite{baghery2013optimal,bayraktar2011proving,bayraktar2013multidimensional} in which conditions for a HJBI equation are proved for stochastic differential games of control and stopping in which the controller however uses continuous controls in the absence of fixed minimal costs.

\section*{The Dynamics: Canonical Description}\

% We suppose that the uncontrolled passive state evolves according to a stochastic process $X:[0,T]\times \Omega \to S\subset  \mathbb{R}^p,\; (p\in \mathbb{N})$, which is a diffusion on $(\mathcal{C}([0,T]; \mathbb{R}^p ),(\mathcal{F}_{{(0,s)}_{s\in [0,T] }}, \mathcal{F}, \mathbb{P}_{0})$ that is, the state process obeys the following SDE:
% \begin{align} dX_s^{t_0,x_{0}}=\mu( s,X_s^{t_0,x_{0}} )ds+\sigma(s,X_s^{t_0,x_{0}} )dB(s)+\int  \gamma(X_{s-},z) \tilde{N}(ds,dz)  ,\; X_{t_0}^{t_0,x_{0}}:= x_{0},&\nonumber
% \\\label{uncontrolledstateprocess}
%  \mathbb{P}-{\rm a.s.}\;\forall s\in [0,T],\; \forall (t_0,x_{0})\in [0,T]\times S;& 
% \end{align}
\noindent We suppose that the uncontrolled passive state evolves according to a stochastic process $X:[0,T]\times \Omega \to S\subset  \mathbb{R}^p,\; (p\in \mathbb{N})$, which is a diffusion on $(\mathcal{C}([0,T]; \mathbb{R}^p ),(\mathcal{F}_{{(0,s)}_{s\in [0,T] }}, \mathcal{F}, \mathbb{P}_{0})$ that is, the state process obeys the following SDE:
\begin{align} dX_s^{t_0,x_{0}}=\mu( s,X_s^{t_0,x_{0}} )ds+\sigma(s,X_s^{t_0,x_{0}} )dB_s,\; X_{t_0}^{t_0,x_{0}}:= x_{0},\;\mathbb{P}-{\rm a.s.}&\nonumber
\\\label{uncontrolledstateprocess}
 \forall s\in [0,T],\; \forall (t_0,x_{0})\in [0,T]\times S;& 
\end{align}
% where $B(s)$ is an $m-$dimensional standard Brownian motion, $\tilde{N}(ds,dz)=N(ds,dz)-\nu(dz)ds$ is a compensated Poisson random measure where $N(ds,dz)$ is a jump measure and $\nu(\cdot):= \mathbb{E}[N(1,\cdot)]$ is a L\'{e}vy measure. Both $\tilde{N}$ and $B$ are supported by the filtered probability space and $\mathcal{F}$ is the filtration of the probability space $(\Omega ,\mathbb{P},\mathcal{F}=\{\mathcal{F}_s\}_{s\in [0,T] } )$. We assume that $N$ and $B$ are independent.
where $B_s$ is an $m-$dimensional standard Brownian motion which supported by the filtered probability space and $\mathcal{F}$ is the filtration of the probability space $(\Omega ,\mathbb{P},\mathcal{F}=\{\mathcal{F}_s\}_{s\in [0,T] } )$. 

% We assume that the functions $\mu:[0,T]\times S\to S,\; \sigma:[0,T]\times S\to \mathbb{R} ^{p\times m}$ and $\gamma:\mathbb{R}^p\times \mathbb{R}^l\to \mathbb{R} ^{p\times l}$  are deterministic, measurable functions that are Lipschitz continuous and satisfy a (polynomial) growth condition so as to ensure the existence of (\ref{uncontrolledstateprocess}) \cite{ikeda2014stochastic} (see assumptions  \hyperlink{A.1.1.}{A.1.1.} \&  \hyperlink{AA.2.}{A.2.}).
We assume that the functions $\mu:[0,T]\times S\to S$ and  $\sigma:[0,T]\times S\to \mathbb{R} ^{p\times m}$ are deterministic, measurable functions that are Lipschitz continuous and satisfy a (polynomial) growth condition so as to ensure the existence of (\ref{uncontrolledstateprocess}) \cite{ikeda2014stochastic} (see assumptions  \hyperlink{A.1.1.}{A.1.1.} \&  \hyperlink{AA.2.}{A.2.}).

% As in \cite{chen2013impulse}, we note that the above specification of the filtration ensures stochastic integration and hence, the controlled jump-diffusion is well defined (this is proven in  \cite{stroock2007multidimensional}).
We note that the above specification of the filtration ensures stochastic integration and hence, the controlled diffusion is well defined (this is proven in  \cite{stroock2007multidimensional}).

% The generator of $X$ (the uncontrolled process) acting on some function\\ $\psi\in \mathcal{C}^{1,2} (\mathbb{R}^l,\mathbb{R}^p)$ is given by: 
% \begin{equation} \mathcal{L}\phi(\cdot,x)=\sum_{i=1}^p   \mu_i (x)    \frac{\partial \phi}{\partial x_i}+\frac{1}{2} \sum_{i,j=1}^p  (\sigma \sigma^T )_{ij} (x)    \frac{ \partial^2\phi}{\partial x_i\partial x_{j} }+I\phi(\cdot,x)	
% \label{generator}
% \end{equation}
The generator of $X$ (the uncontrolled process) acting on some function\\ $\phi\in \mathcal{C}^{1,2} (\mathbb{R}^l,\mathbb{R}^p)$ is given by: 
\begin{equation} 
\mathcal{L}\phi(\cdot,x)=\sum_{i=1}^p   \mu_i (x)    \frac{\partial \phi}{\partial x_i}(\cdot,x)+\frac{1}{2} \sum_{i,j=1}^p  (\sigma \sigma^T )_{ij} (x)    \frac{ \partial^2\phi}{\partial x_i\partial x_{j} }(\cdot,x).
\label{generator}
\end{equation}
% where $I$ is the integro-differential operator defined by:
% \begin{equation} 
% I\phi(\cdot,x):= \sum_{j=1}^{l} \int_{\mathbb{R}^p}  \{\phi(\cdot,x+ \gamma^{j}(x,z_j))-\phi(\cdot,x)-\nabla\phi(\cdot,x)  \gamma^{j} (x,z_{j} )\}  \nu_{j} (dz_{j}),\;{\forall  x\in \mathbb{R}^p.}\label{integro_differential_operator} 
% \end{equation}

In this game there are two players, player I and player II. Throughout the horizon of the game, each player incurs a cost which is a function of the value of the state process. Let the set  $\mathcal{T}$ be a given family of $\mathcal{F}-$measurable stopping times; at any point in the game $\rho \in \mathcal{T}$, player II can choose to terminate the game at which point the state process is stopped and both players receive a terminal cost (reward). Player I influences the state process using impulse controls $u\in \mathcal{U}$ where $u(s)=\sum_{j\geq 1}\xi_j \cdot 1_{\{\tau_j\leq T \}}  (s)$ for all $0\leq t_0< s\leq T$  where $\xi_1,\xi_2,\ldots\in \mathcal{Z}\subset S$ are impulses that are executed at $\mathcal{F}$-measurable stopping times $\{\tau_i\}_{i\in\mathbb{N}}$ where $0\leq t_0\leq \tau_1< \tau_2< \dots <$ and where the set $\mathcal{U}$ is a convex cone that defines the set of player I controls. We assume that the impulses $\xi_j \in \mathcal{Z}$ are $\mathcal{F}-$measurable for all $j \in \mathbb{N}$.  Hence, if an impulse $\zeta \in \mathcal{Z}$ determined by some (admissible) policy $u\in\mathcal{U}$ is applied at some $\mathcal{F}-$measurable stopping time $\tau:\Omega \to [0,T]$ when the state is $x'=X^{t_0,x_0,\cdot} (\tau^-)$, then the state immediately jumps from $x'=X^{t_0,x_0,\cdot} (\tau^-)$ to $X^{t_0,x_0,u} (\tau)=\Gamma (x',\zeta)$ where $\Gamma :S\times \mathcal{Z}\to S$ is called the impulse response function.

Hence, the control $u=u(s,\omega)\in\mathcal{U};\; s\in [0,T],\omega\in\Omega$ exercised by player I is a stochastic process that modifies the state process directly.
% Moreover, the player I control can be written in the form $u=\tilde{f}_1(s,X_s)$ for any $s\in[0,T]$ where $\tilde{f}_1:[0,T]\times S\to U$, $U\subset \mathbb{R}^p$ and $\tilde{f}_1$ is some measurable map  w.r.t. $\mathcal{F}$. 

For notational convenience, we use $u=[\tau_j,\xi_j ]_{j\geq 1} $ to denote the player I control policy $u=\sum_{j\geq 1}\xi_j  \cdot1_{\{\tau_j\leq T \}}  (s)\in \mathcal{U}$.

% The evolution of the state process with actions is given by the following:
% \begin{align}\nonumber
% \hspace{-1.8 mm}X_r^{t_0,x_0,u}=x_0+\int_{t_0}^{r\wedge\rho}\hspace{-1.5 mm}\mu(s,X^{t_0,x_0,u}_s)ds+\int_{t_0}^{r\wedge\rho}\hspace{-1.5 mm}\sigma(s,X^{t_0,x_0,u}_s)dB(s)
% +\sum_{j\geq 1}\xi_j  \cdot 1_{\{\tau_j\leq r\wedge \rho \}}  (r)&
% \\+\int_{t_0}^{r\wedge\rho}\hspace{-1 mm}\int\gamma (X^{t_0,x_0,u}_{s-},z) \tilde{N}(ds,dz),&\nonumber 
% \\
% \forall r\in [0,T];\;\forall (t_0,x_0)\in [0,T]\times S,\; \mathbb{P}-{\rm a.s.}&
% \label{controlledstateprocessch1player}
% \end{align}
The evolution of the state process with actions is given by the following:
\begin{align}\nonumber
\hspace{-1.8 mm}X_r^{t_0,x_0,u}=x_0+\int_{t_0}^{r\wedge\rho}\hspace{-1.5 mm}\mu(s,X^{t_0,x_0,u}_s)ds+\int_{t_0}^{r\wedge\rho}\hspace{-1.5 mm}\sigma(s,X^{t_0,x_0,u}_s)dB_s
+\sum_{j\geq 1}\xi_j  \cdot 1_{\{\tau_j\leq r\wedge \rho \}}  (r)&\nonumber 
\\
\forall r\in [0,T];\;\forall (t_0,x_0)\in [0,T]\times S,\; \mathbb{P}-{\rm a.s.}&
\label{controlledstateprocessch1player}
\end{align}

Without loss of generality we assume that $X^{t_0,x_0,\cdot}_s=x_0$ for any $s\leq t_0$.

Player I has a cost function which is also the player II gain (or profit) function. The corresponding payoff function is given by the following expression which player I (resp., player II) minimises (resp., maximises):  
\begin{align}\nonumber 
J [t_0, x_0  ;u,\rho ]= \mathbb{E}\Bigg[\int_{t_0}^{\tau_s  \wedge\rho} f (s, X_s^{t_0, x_0  ,u } ) ds  &+ \sum_{m\geq 1}  c (\tau_m  , \xi _m  )  \cdot 1_{\{\tau_m  \leq  \tau_S  \wedge\rho \}}
\\&\begin{aligned}+G (\tau_S  \wedge\rho, X_{\tau_S  \wedge\rho }^{t_0, x_0  ,u } )1_{\{\tau_S  \wedge\rho<\infty\}}\Bigg],&\\ 
\forall (t_0,x_0)\in[0,T]\times S,& 
\end{aligned}
\end{align}
where $\tau_S:\Omega\to [0,T] $ is some random exit time, i.e. $\tau_S(\omega):=\inf\{s\in [0,T]|X_s^{t_0, x_0  ,\cdot }\in S\backslash A;\;\omega\in\Omega\}$ where $A$ is some measurable subset of $S$, at which point $\tau_S$ the game is terminated. The functions $f:[0,T]\times S\to \mathbb{R} , G:[0,T]\times S\to \mathbb{R}$ are the running cost function and the bequest function respectively and the function $c:[0,T]\times\mathcal{Z}\to\mathbb{R}$ is the intervention cost function. 

We assume that the function $G$ satisfies the condition $\underset{s\to\infty}{\lim}G(s,x)=0$ for any $x\in S$. Functions of the form $G(s,x)\equiv e^{-\delta s}\bar{G}(x)$ for some $\delta>0$ with $\bar{G}:|\bar{G}(x)|<\infty$ satisfy this condition among others.

The results contained in this paper are built exclusively under the following set of assumptions unless otherwise stated:

\section*{Standing Assumptions}
\

\noindent A.1.1. Lipschitz Continuity \hypertarget{A.1.1.}{}

We assume there exist real-valued constants $c_{\mu},c_{\sigma}>0$ s.th. $\forall  s\in [0, T], \forall  x,y\in S$ we have:
\begin{align*}
|\mu(s,x)-\mu(s,y)|&\leq c_{\mu} |x-y|\\
|\sigma(s,x)-\sigma(s,y)|&\leq c_{\sigma} |x-y|.
\end{align*}
A.1.2. Lipschitz Continuity \hypertarget{A.1.2.}{}

We also assume the Lipschitzianity of the functions $f,G$ and $\phi$ and the so that we assume the existence of real-valued constants $c_f,c_G>0$ s.th.  $\forall  s\in [0, T], \forall  (x,y) \in S$ we have for $R\in \{f,G\}$:
\begin{equation*}
|R(s,x)+R(s,y)|\leq c_R |x-y|.
\end{equation*}
 \noindent A.2. Growth Conditions \hypertarget{AA.2}{}

We assume the existence of a real-valued constants $d_{\mu},d_{\sigma}>0$ s.th. $\forall  (s,x)\in [0, T]\times S$ we have:
\begin{align*}
|\mu(s,x)|&\leq d_{\mu} (|1+|x|^\rho |)
\\
|\sigma(s,x)|&\leq d_{\sigma} (|1+|x|^\rho |).
\end{align*}
We also make the following assumptions on the cost function $c:[0, T]\times S\to \mathbb{R}$:\\
A.3.\ \hypertarget{AA.3.}{}

Let $\tau,\tau' \in \mathcal{T} $ be $\mathcal{F}-$measurable stopping times s.th. $0\leq \tau<\tau'\leq T$ and let $\xi,\xi' \in \mathcal{Z}$ be measurable impulse interventions. Then we assume that the following statements hold:
	\begin{align} c(\tau,\xi+\xi' )&\leq c(\tau,\xi)+c(\tau,\xi' ),\\
	  c(\tau,\xi)&\geq c(\tau',\xi).
\end{align}
A.4.\ \hypertarget{AA.4.}{}

We also assume that there exists a constant $\lambda_c>0$ s.th.  $\inf_{\xi \in \mathcal{Z}}c(s,\xi)\geq \lambda_c\forall  s \in [0,T]$ where $\xi \in \mathcal{Z}$ is a measurable impulse intervention.\

Assumptions  \hyperlink{A.1.1.}{A.1.1.} and  \hyperlink{AA.2.}{A.2.} ensure the existence and uniqueness of a solution to (\ref{uncontrolledstateprocess}) (c.f. \cite{ikeda2014stochastic}). Assumption  \hyperlink{A.1.2.}{A.1.2.} is required to prove the regularity of the value function (see for example \cite{cosso2013stochastic} and for the single-player case, see for example \cite{davis2010impulse}). Assumption  \hyperlink{AA.3.}{A.3.} (i) (subadditivity) is required in the proof of the uniqueness of the value function. Assumption  \hyperlink{AA.3.}{A.3.} (ii) (the player cost function is a decreasing function in time) may be interpreted as a discounting effect on the cost of interventions.  Assumption  \hyperlink{AA.3.}{A.3.} (ii) was introduced (for the two-player case) in \cite{yong1994zero} though is common in the treatment of single-player case problems (e.g. \cite{davis2010impulse,chen2013impulse}). Assumption  \hyperlink{AA.4.}{A.4.} is integral to the definition of the impulse control problem.

Throughout the script we adopt the following standard notation (e.g. \cite{cardaliaguet2010introduction,chen2013impulse,stroock2007multidimensional}):

\section*{Notation}

Let $\Omega$  be a bounded open set on  $\mathbb{R}^{p+1}$. Then we denote by:
$\bar{\Omega}$  - The closure of the set $\Omega$.\\
$Q(s,x;R)={{(s',x' ) \in \mathbb{R} ^{p+1}:\max |s'-s|^{\frac{1}{2}}  ,|x'-x|  }<R,s'<s}$. \\
$\partial \Omega$  - The parabolic boundary $\Omega$  i.e. the set of points $(s,x) \in \bar{\mathcal{S}}$ s.th. $R>0, Q(s,x;R)\not\subset\bar{\Omega}$.\\
$\mathcal{C}^{\{1,2\}} ([0, T],\Omega )=\{h \in C^{\{1,2\}} (\Omega ): \partial_s h, \partial_{x_i,x_{j} } h \in C(\Omega )\}$, where  $\partial_s$ and  $\partial_{x_{i}, x_{j}}$ denote the temporal differential operator and second spatial differential operator respectively.\\
$\nabla\phi=(\frac{\partial \phi}{\partial x_1 },\ldots,\frac{\partial \phi}{\partial x_p})$ - The gradient operator acting on some function $\phi \in C^1 ([0,T]\times \mathbb{R}^p)$.\\
$|\cdot|$   - The Euclidean norm to which $\langle x,y \rangle$  is the associated scalar product acting between two vectors belonging to some finite dimensional space.
\section{Statement of Main Results}\hypertarget{II}{}
In this paper, we prove two main results for the game that characterise the conditions for a HJBI equation in both zero-sum and non-zero-sum impulse controller-stopper stochastic differential games.
	
We prove that the stochastic differential game in which one of the players uses impulse controls and the other player chooses when to end the game admits a value. 

	We prove that the value of the game satisfies a double obstacle quasi-integro-\\variational equality and is a unique viscosity solution to a HJBI equation. 

In particular, in section \hyperlink{V}{V} we show that equality (\ref{viscvaluefunctionequalitytheorem}) holds by firstly showing that $V^+ \text{(resp.; }V^- ) $ is a viscosity supersolution (resp.; subsolution) to the following non-linear obstacle problem: \begin{align}
\{\max\{\min[-\frac{\partial V}{\partial s}-\mathcal{L}V-f,V-G ],V-\mathcal{M}V \}=0 \nonumber\\
V(X^{t,x,u} (\tau_s\wedge\rho))=G(X^{t,x, u} (\tau_s\wedge\rho)) . 	\label{viscp1obstacle}\end{align}
where $\mathcal{L}$ is the local stochastic generator operator associated to the process and $\mathcal{M}$ is the non-local intervention operator.

To our knowledge, this is the first game that involves impulse controls in which the role of one of the players is to stop the game at a desirable point. 
\section*{Organisation}

In section \hyperlink{III}{III}, we provide some of the technical definitions in order to provide a description of the game. Here, we also introduce the underlying concepts required to study the game. In section \hyperlink{IV}{IV}, we prove some preliminary results that underpin the main analysis which is performed in section \hyperlink{V}{V} --- in particular, we prove that the upper and lower value functions are regular and bounded. Where it is of no detriment to the main body of ideas, we postpone some of the technical proofs of the section to the appendix. In section \hyperlink{V}{V}, we introduce the viscosity theory framework --- here we show by way of a dynamic programming principle and comparison principle, that the value of the game is a unique viscosity solution to a HJBI equation. Lastly, the paper is finalised by an \hyperlink{VI}{appendix} to which some of the technical proofs from sections \hyperlink{IV}{IV} - \hyperlink{V}{V} are deferred.
\section*{Definitions} \hypertarget{III}{}
We now give some definitions which we shall need to describe the system dynamics modified by impulse controls:

\begin{definition}\label{Definition 1.2.3.}
Denote by $\mathcal{T}_{(t,\tau')}$  the set of all $\mathcal{F}-$measurable stopping times in the interval $[t,\tau']$, where $\tau'$ is some stopping time s.th. $\tau' \leq T$. If $\tau'=T$ then we denote by $\mathcal{T}\equiv \mathcal{T}_{(0,T)}$. Let $u=[\tau_j,\xi_j]_{j\in\mathbb{N}}$ be a control policy where $\{\tau_j\}_{j\in\mathbb{N}}$ and $\{\xi_j\}_{j\in\mathbb{N}}$ are $\mathcal{F}_{\tau_j}-$ measurable stopping times and interventions respectively, then we denote by $ \mu_{[t,\tau]} (u)$ the number of impulses the controller (player I) executes within the interval $[t,\tau]$ under the control policy $u\in\mathcal{U}$ for some $\tau \in \mathcal{T}$.\end{definition} 

\begin{definition}\label{Definition 1.2.4.}
Let $u\in\mathcal{U}$ be a player I impulse control. We say that an impulse control is admissible on $[0,T]$ if either the number of impulse interventions is finite on average that is to say we have that:
\begin{equation}\mathbb{E}[\mu_{[0,T]} (u)]<\infty	 \end{equation}
or if $\mu_{[0,T]} (u)=\infty\implies\lim_{j\to \infty }\tau_j=\infty$.
\end{definition}
We shall hereon use the symbol $\mathcal{U}$ to denote the set of admissible player I controls. For controls $u\in \mathcal{U}$ and $u'\in \mathcal{U}$, we interpret the notion $u\equiv u'$ on $[0,T]$ iff $\mathbb{P}(u=u'$ a.e. on $[0,T])=1$. 

\begin{definition}\label{Definition 1.2.5.}
Let $u(s)=\sum_{j\geq 1}\xi_j  \cdot 1_{\{\tau_j\leq T \}}  (s) \in \mathcal{U}$ be a player I impulse control defined over $[0,T]$, further suppose that $\tau:\Omega \to [0,T]$ and $\tau':\Omega \to [0,T]$ are two $\mathcal{F}-$measurable stopping times with $\tau\geq s>\tau'$, then we define the restriction $u_{[\tau',\tau ]} \in \mathcal{U}$ of the impulse control $u(s)$ to be $u(s)=\sum_{j\geq 1}\xi_{\mu_{]t_0,\tau)}(u)+j}  \cdot 1_{\{\tau_{\mu_{]t_0,\tau)}(u)+j} \geq s\geq \tau' \}}  (s)$. 
\end{definition}
\section*{Strategies}

A player strategy is a map from the other player's set of controls to the player's own set of controls. An important feature of the players' strategies is that they are non-anticipative --- neither player may guess in advance, the future behaviour of other players given their current information.

We formalise this condition by constructing non-anticipative strategies which were used in the viscosity solution approach to differential games in \cite{fleming1989existence}. Non-anticipative strategies were introduced by \cite{roxin1969axiomatic,elliott1972existence,varaiya1967existence}. Hence, in this game, one of the players chooses his control and the other player responds by selecting a control according to some strategy.
\begin{definition}\label{Definition 1.1.1.}
A non-anticipative strategy on $[0,T]$ for player I is a measurable map which we shall denote by $\alpha $ s.th. $\alpha :\mathcal{T}\to \mathcal{U}$ and for any stopping time $\tau:\Omega \to [0,T]$ and any $v_1,v_2  \in \mathcal{T}$  with $v_1\equiv v_2$ on $[t,\tau]$ we have that $\alpha (v_1 )\equiv \alpha (v_2 )$ on $[t,\tau]$.
\end{definition} 

We define the player II non-anticipative strategy $\beta :\mathcal{U}\to\mathcal{T}$ analogously. Hence, $\alpha$  and $\beta$  are Elliott-Kalton strategies.

Following the notation in \cite{cardaliaguet:double_obstacle}, we denote the set of all non-anticipative strategies  over the time horizon $[0,T]$ for player I (resp., player II) by $\mathcal{A}_{(0,T)}$ (resp., $\mathcal{B}_{(0,T )}$).

% It can be proven (e.g. Lemma 1.1 pg. 5 in \cite{cardaliaguet2009stochastic}) that associated to the strategies $\alpha \in \mathcal{A}_{(0,T)}$ and $\beta \in \mathcal{B}_{(0,T )}$ are a set of unique pair of controls $(u,v)\in \mathcal{U}\times \mathcal{V}$. Thus, in the following game we assume that for all $\alpha \in \mathcal{A}_{(0,T)}$ and $\beta \in \mathcal{B}_{(0,T )}$ there exists a unique set of controls $(u,v)\in \mathcal{U}\times \mathcal{V}$ s.th. $\mathbb{P}-$a.s.:
% \begin{equation}
% (u,v)=(\alpha (\cdot ,X_\cdot ^{t_0,x_0,u,v},v),\beta (\cdot ,X_\cdot ^{t_0,x_0,u,v},u)),	\;\forall (t_0,x_0)\in [0,T]\times S \text{ a.e. on } [0,T]
% \end{equation}	
\begin{remark}\label{remark 1.1.2.}
The intuition behind Definition \ref{Definition 1.1.1.} is as follows: suppose player I employs $u_1\in\mathcal{U}   $ and the system follows a path $\omega$  and that player II employs the strategy $\beta \in \mathcal{B}_{(0,T )}$ against the control $u_1$. If in fact player II cannot distinguish between the control $u_1$ and some other player I control $u_2\in\mathcal{U}   $ then controls $u_1$ and $u_2$ induce the same response from the player II strategy that is to say $\beta (u_1 )\equiv \beta (u_2)$.
\end{remark} 
The notion of non-anticipative strategies is derived from the fact that the strategies defined in Definition \ref{Definition 1.1.1.} exclude the possibility of a player exploiting future information of their opponent's control modifications. 

Note that when $\mathcal{U}$ is a singleton the game is degenerate and collapses into a classical optimal stopping problem for player II with a value function and solution as that in ch.3 in \cite{oksendal2005applied}. Similarly, when $\mathcal{T}$ is a singleton the game collapses into a classical impulse control problem for player I with a value function and solution as that in ch.6 in \cite{oksendal2005applied}.

The following definition is a key object in the analysis of impulse control models:
\begin{definition}\label{Definition 2.1.1.}
Let $\tau\in\mathcal{T}$, we define the [non-local] intervention operator $\mathcal{M}:\mathcal{H}\to \mathcal{H}$ acting at a state $X(\tau)$ by the following expression:
\begin{equation}
\mathcal{M} \phi(\tau,X(\tau)):=\inf_{z\in \mathcal{Z}}[\phi(\tau,\Gamma (X(\tau^-),z))+c(\tau, z)\cdot 1_{\{\tau\leq T \}}  ],\label{intervention_operator_definition_equation} \end{equation}
for some function $\phi :[0,T]\times S\to \mathbb{R}$ and $\Gamma : S\times \mathcal{Z}\to  S$ is the impulse response function.
\end{definition}

Of particular interest is the case when the intervention operator is applied to the value function $\mathcal{M}V(\cdot,x)$ --- a quantity which represents the value of a strategy when the controller performs an optimal intervention then behaves optimally thereafter given an immediate optimal intervention taken at a state $x\in S$. The intuition behind (\ref{intervention_operator_definition_equation}) is as follows: suppose at time $\tau^-$ the system is at a state $X(\tau^-)$ and an intervention $z\in\mathcal{Z}$ is applied to the process, then a cost of $c(\tau,z)$ is incurred and the state then jumps from $X(\tau^-)$ to $\Gamma(X(\tau^-),z)$. If the controller acts optimally thereafter, the cost of this strategy, starting at state $\Gamma(X(\tau^-),z)$ is $V(\tau,\Gamma(X(\tau^-),z)+c(\tau,z)$. Lastly, choosing the action that minimises costs leads to $\mathcal{M}V$.

\begin{remark}\label{remark 1.2.7.}
We note that whenever it is optimal for the controller to intervene, $\mathcal{M}V=V$ since the value function describes the player payoff under optimal behaviour. However, at any given point an immediate intervention may not be optimal, hence the following inequality holds pointwise:
\begin{align}
\mathcal{M}V(s,x)&\geq V(s,x), 	\label{intervention_operator_inequality_one_player} \qquad \hfill \forall  (s,x)\in [0,T]\times S.     
\end{align} 
\end{remark}
\section{Preliminaries} \hypertarget{IV}{}
 Given the remarks of section \hyperlink{III}{III}, we now define the value functions of the game. As in \cite{fleming1989existence}, we define the value functions in terms of Elliot-Kalton strategies introduced in \cite{elliott1972existence}:
\\
\section*{Value Functions}

\begin{align}
V^- (t,x)&= \inf_{\alpha \in \mathcal{A}_{(0,T)}}  \sup_{\rho' \in \mathcal{T}}  J[t,x;\alpha(\rho'),\rho' ];\hspace{2 mm} \label{viscvaluefunctionminus}
\\
V^+(t,x)&=\sup_{\beta \in \mathcal{B}_{(0,T)}}    \inf_{u' \in \mathcal{U}}  J[t,x;u',\beta(u' )],\qquad \forall (t,x) \in [0,T]\times S,
\label{viscvaluefunctionplus} 
\end{align}
where $\alpha\in\mathcal{A}_{(0,T)}$ and $\beta\in\mathcal{B}_{(0,T)}$ are player I and player II (non-anticipative) strategies respectively (c.f. Definition \ref{Definition 1.1.1.}), $\mathcal{T}$ is a given family of $\mathcal{F}-$measurable stopping times and $\mathcal{U}$ is the set of player I admissible impulse controls. 

We refer to $V^-$ and $V^+$ as the upper and lower value functions respectively.

We say that the value of the game exists if we can commute the supremum and infimum operators in (\ref{viscvaluefunctionplus}) and (\ref{viscvaluefunctionminus}) where after we can deduce the existence of a function $V \in \mathcal{H}$ s.th. $V(t,x)\equiv V^- (t,x)=V^+ (t,x)$ for any $(t,x)\in [0,T]\times S$. We use the notation $V^{\pm}$ to mean any element drawn from the set $ \{V^+,V^-\}$.
\

\begin{theorem}\label{viscvaluefunctionequalitytheorem}
The value of the game exists and is given by: 
\begin{equation}
V(t,x)=V^- (t,x)=V^+(t,x), \quad \forall (t,x) \in [0,T]\times S. 	
\end{equation}
\end{theorem}

\begin{remark}\label{remark 2.1.2.}
By definition of the value functions, we automatically have:
\begin{equation}
V^- (t,x)\geq V^+(t,x),\quad \forall (t,x) \in [0,T]\times S. 	\label{viscvaluefunctioninequalitystandard}
\end{equation}
\end{remark}

To prove Theorem \ref{viscvaluefunctionequalitytheorem}, it therefore remains to establish the reverse inequality of (\ref{viscvaluefunctioninequalitystandard}). 

We now prove the regularity of the value functions associated to the game. Related results can be found in \cite{cosso2013stochastic} where the regularity (Lipschitzianity in the spatial variable, H\"{o}lder continuity in time) for the corresponding two-player impulse control game with an uncontrolled process.\

\begin{lemma}\label{Lemma 2.1.3.}
The functions $V^-$  and $V^+$ can be equivalently expressed by the following: 
\begin{align}
V^- (t,x)&= \inf_{\alpha \in \bar{\mathcal{A}}_{(t,T)}}  \sup_{\rho \in \mathcal{T}\backslash \{t\}}  J[t,x;\alpha(\rho),\rho],\label{viscvaluefunctionminusredefine}	
\\
V^+(t,x )&=\sup_{\beta \in \hat{\mathcal{B}}_{(t,T)}}\inf_ {u \in \bar{\mathcal{U}}_{(t,T )}}  J[t,x;u,\beta(u)],\qquad \forall  (t,x) \in [0,T]\times S,\label{viscvaluefunctionplusredefine}	
\end{align}
where $\bar{\mathcal{U}}_{(t,T )}$ is the set of player I admissible controls which have no impulses at time $t$ and correspondingly, $\bar{\mathcal{A}}_{(t,T)}$ (resp., $\hat{\mathcal{B}}_{(t,T)} $ is the set of all player I (resp., player II) non-anticipative strategies with controls drawn from the set $\bar{\mathcal{U}}_{(t,T )}$ (resp., $\mathcal{T} \backslash\{t\}$).
\end{lemma}

\begin{proof}

The proof of the lemma is similar to that of Lemma 3.3 in \cite{cosso2013stochastic} with some modifications: 

The main idea is to prove that for all $ (t,x ) \in [0,T]\times S$, there exists a control $\bar{u} \in \mathcal{U}\backslash \bar{\mathcal{U}}_{(0,T )} $ and an $\mathcal{F}-$ measurable stopping time $\bar{\rho} \in \mathcal{T}\backslash {t'}$ where $\bar{\mathcal{U}}_{(0,T )}\subset \mathcal{U}$ is the set of admissible impulse controls $\mathcal{U}$ which excludes impulses at time $t'$, for which the following inequality holds $\forall  u \in \mathcal{U}$, given some  $\mathcal{F}-$measurable stopping time $\rho \in \mathcal{T}$ and for some $\epsilon >0$: \begin{equation}
\left|J[t',x;u,\beta(u)]-J[t',x;\bar{u},\bar{\beta}(\bar{u})]\right|+\left|J[t',x;\alpha(\rho),\rho]-J[t',x;\bar{\alpha}(\bar{\rho}),\bar{\rho}]\right|\leq \epsilon .
\end{equation}

We prove the result for the case in which player I exercises only one intervention at the point $t$ since the extension to multiple interventions is straightforward.

W.l.o.g., we can employ the following short-hands $\beta(u)\equiv \rho \in \mathcal{T}, \bar{\beta}(\bar{u})\equiv \bar{\rho} \in \mathcal{T}_{t' }, \alpha(\rho)\equiv u \in \mathcal{U}$ and $\bar{\alpha}(\bar{\rho})\equiv \bar{u} \in \bar{\mathcal{U}}_{(0,T )} $. The result is proven by constructing the following control and stopping times:
\begin{equation*}\noindent
u_n=\xi\cdot 1_{[\tau_n,T]} +u',\end{equation*}
where $\tau_n=(\tau+\frac{1}{n}) \cdot 1_{\{\tau=t'\}}+\tau\cdot 1_{\{\tau>t'\}}$
and
\begin{equation*} 
\rho_n=\left(\rho+\frac{1}{n}\right) \cdot 1_{\{\rho=t'\}}+\rho\cdot 1_{\{\rho>t'\}},\end{equation*}
where $u'=\sum_{j\geq  1} \xi_j \cdot 1_{[\tau_j,T]} $.

Hence, we find that:
\begin{align*}
&\qquad\qquad\qquad\qquad\qquad\qquad\qquad J[t',x_0;u,\rho]-J[t',x_0;u_n,\rho_n]
\\&\begin{aligned}
= \mathbb{E}\Bigg[\int_{t'}^{\rho\wedge \tau_S}&f(s,X_s^{t' , x_0  ,u} )ds  - \int_{t'}^{\rho_n  \wedge \tau_S}f(s,X_s^{t' , x_0  , u_n} )ds  - \sum_{j\geq  1}  c (\tau_j  , \xi_j  ) \cdot 1_{\{\rho\leq  \tau_j  < \rho_n  \}}
\\&
+\left(G(\rho\wedge \tau_S,X_{\rho\wedge\tau_S}^{t' , x_0  ,u} )-G(\rho_n\wedge \tau_S,X_{\rho_n\wedge \tau_S}^{t' , x_0  , u_n } )\right)\cdot1_{\{\tau_S\wedge\rho_n<\infty\}}\Bigg] 
\end{aligned}
\\
&\begin{aligned}
=\mathbb{E}\Bigg[\int_{t'}^{\rho\wedge\tau_S} \big(f(s,X_s^{t',x_0,u})&-f(s,X_s^{t',x_0,u_n } )\big)ds-\int_{\rho\wedge\tau_S}^{\rho_n\wedge\tau_S} f(s,X_s^{t',x_0,u_n } )ds
\\&\begin{aligned}  +\left(G(\rho\wedge\tau_S,X_{\rho\wedge\tau_S}^{t',x_0,u} )-G(\rho_n\wedge\tau_S,X_{\rho_n\wedge\tau_S}^{t',x_0,u_n} )\right)\cdot1_{\{\tau_S\wedge\rho_n<\infty\}}
&\\-\sum_{j\geq  1} c(\tau_j,\xi_j ) \cdot 1_{\{\rho\leq \tau_j<\rho_n \}}\Bigg].&
\end{aligned}
\end{aligned}
\end{align*}
 
 We now readily observe that $X_s^{t',x_0,u_n }\to X_s^{t',x_0,u},\; \mathbb{P}-$a.s. Additionally, by construction, $\rho_n\to \rho$ as $n\to  \infty, \mathbb{P}$-a.s. hence, after invoking the dominated convergence theorem, we can deduce the existence of an integer $N\geq 1$ s.th. $\forall  \epsilon >0$ and $\forall  n\geq N$ such that: 
\begin{equation}
J[t',x_0;u,\rho]-J[t',x_0;u_n,\rho_n]\leq \epsilon . \end{equation}
 The proof can be extended to multiple impulse case (at time $t'$) straightforwardly after employing conditions A3 (iii) and A3 (iv) where after the proof easily reduces to the single impulse case.
\end{proof}

Given Lemma \ref{Lemma 2.1.3.}, we can now seek to prove that the upper and lower value functions associated to the game are Lipschitz continuous in the spatial variable and $\frac{1}{2}$-H\"{o}lder continuous in time, that is, we prove the following proposition:
\begin{proposition}\label{Proposition 2.1.4.}
We can deduce the existence of constants $c_1,c_2>0$ s.th. the following results hold:  
\renewcommand{\theenumi}{\roman{enumi}}
 \begin{enumerate}
\item	$|V^-  (t, x')- V^-  (t,x )|+ |V^+  (t, x')- V^+  (t,x )|\leq  c_1  |x'-x |,$
	\item$|V^-  (t',x )- V^-  (t,x) |+ |V^+  (t',x )- V^+  (t,x) |\leq  c_2  |t'-t |^{\frac{1}{2}}$. 
\end{enumerate} 
\qquad\qquad\qquad\qquad\qquad\qquad\qquad\qquad\qquad\qquad\qquad $\forall (t,x),(t',x') \in [0,T]\times S$.
\end{proposition} 
Proposition \ref{Proposition 2.1.4.} establishes an important property of the game --- small changes in the input variables of the value functions lead to small changes in the game. This result is crucial for deriving the DPP for the game which describes the behaviour of the value function under infinitesimal variations.

\begin{proof}
We separate the proof into two parts, proving the spatial Lipschitzianity (i) first, then the temporal $\frac{1}{2}$-H\"{o}lder-continuity (ii) last.

To show that the value functions are Lipschitz continuous in the spatial variable, it suffices to show that the property is satisfied for the function $J$.

The proof is straightforward since the result follows as an immediate consequence of the Lipschitzianity of the constituent functions. In particular we have:
 \begin{align*}
 &\left| J[t,x;\cdot,\cdot]-J[t,x';\cdot,\cdot]\right|
 \\&\leq \mathbb{E}\left[\int_t^{\tau_s\wedge\rho} \left| f(s,X_s^{t,x,\cdot} )-f(s,X_s^{t,x',\cdot} )\right|  ds+\left| G(\tau_S\wedge\rho,X_{\tau_S\wedge\rho}^{t,x,\cdot} )-G(\tau_S\wedge\rho,X_{\tau_S\wedge\rho}^{t,x',\cdot} )\right|\right]
\\&\begin{aligned}
\leq c_f \int_t^{\tau_s\wedge\rho} \mathbb{E}\left[\left| X_s^{t,x,\cdot}-X_s^{t,x',\cdot} \right| \right]ds+c_G\mathbb{E} \left[\left| X_{\tau_S\wedge\rho}^{t,x,\cdot}-X_{\tau_S\wedge\rho}^{t,x',\cdot} \right|\right], 
\\&\hspace{-5 mm}\forall  (t,x' ),(t,x) \in [0,T]\times S,
\end{aligned}
\end{align*}
where $c_f,c_G>0$ are Lipschitz constants for the function $f$ and $G$ respectively (see assumption \hyperlink{A.1.2.}{A.1.2.}).
Therefore, as an immediate consequence of Lemma \hyperlink{Lemma A.1}{A.1}, we see that we can deduce the existence of a constant $c>0$ s.th.
\begin{equation}
|J[t,x;\cdot,\cdot]-J[t,x';\cdot,\cdot]|\leq c|x-x' |. \label{viscJlipschitzspatial}	
\end{equation}

We note also that since the constituent functions of $J$ are bounded, $J$ is also bounded; hence by (\ref{viscJlipschitzspatial}) and by Lemma 3.6 in \cite{cardaliaguet2010introduction}, we therefore conclude that:
\begin{equation}
|V^{\pm} (t,x)-V^{\pm} (t,x')|\leq c|x-x'|,	
\end{equation}
for some constant $c>0$ as required.

To show that the value functions are Lipschitz continuous in the temporal variable, we show that (ii) is satisfied by the function $V^+$ with the proof for the function $V^-$ being analogous.

 We firstly note that: 
\begin{align}
V^+ (t',x)-V^+ (t,x)=\sup_{\mu \in \mathcal{B}_{(0,T)}}   \inf_{u \in \mathcal{U}}  J[t',x;u,\mu(u)] -\sup_{\mu \in \mathcal{B}_{(0,T)}}   \inf_{u \in \mathcal{U}}  J[t,x;u,\mu(u)]&\nonumber\\
\forall  (t,x' ),(t,x) \in [0,T]\times S.& 
\label{vminusveqjminusj}
\end{align}
By (i), we can deduce the existence of some $\epsilon -$optimal strategy $\hat{\mu} \in \mathcal{B}_{(0,T)} $\\against $V^+ (t,x) $ s.th. $V^+ (t,x)-\epsilon \leq  \inf_{u \in \mathcal{U}} J[t,x;u,\hat{\mu}(u)],$ $ \forall  (t,x) \in [0,T]\times S$ where $\epsilon >0$ is arbitrary. Hence, by (\ref{vminusveqjminusj}) we have that: 
\begin{align}
V^+ (t',x)-V^+ (t,x)-2\epsilon \leq    \inf_{u \in \mathcal{U}} J[t',x;u,\hat{\mu}(u)] -   \inf_{u \in \mathcal{U}} J[t,x;u,\hat{\mu}(u)].
\end{align} 
Let us now construct the control $u_\epsilon =\sum_{j\geq  1} \xi_j^\epsilon  \cdot 1_{[\tau_j^\epsilon ,T ]} $  which is associated with the strategy $\alpha^\epsilon  \in \mathcal{A}_{(0,T)} $. Let us also construct the control $u'_\epsilon  \in \mathcal{U}_{[t',T]}$  using the following expression: $u'_\epsilon =\sum_{\tau_j^\epsilon \leq t'} \xi_j^\epsilon  \cdot 1_{t'} +\sum_{\tau_j^\epsilon >t'} \xi_j^\epsilon  \cdot 1_{[\tau_j^\epsilon ,T]} $  which is associated to the strategy $\hat{\alpha}^\epsilon$ so that the control $u'_\epsilon$ is simply the control $u_\epsilon  \in \mathcal{U}$ except that the impulse interventions within the interval $[t,t'[$ are now pushed to $t'$.

Now thanks to Lemma \ref{Lemma 2.1.3.}, we have that $|J[t,x;u,\cdot]-J[t,x;\bar{u},\cdot]|<\epsilon$  where $\epsilon >0$ is arbitrary and where $\bar{u} \in \bar{\mathcal{U}}$ is the set of player I admissible controls that have no impulses at time $t$. Hence, by Lemma \ref{Lemma 2.1.3.} and using the $\epsilon -$optimality of the strategy $\hat{\mu} \in \mathcal{B}_{(0,T)} $, we can therefore deduce the following inequality: \begin{equation}
V^+ (t',x)-V^+ (t,x)-3\epsilon \leq J[t',x;u'_\epsilon ,\mu(u'_\epsilon  )] - J[t,x;u_\epsilon ,\hat{\mu}(u_\epsilon  )]. \label{VplusholderproofJminusJ}	\end{equation}
Let us denote by $\hat{\rho}:=\mu(u'_\epsilon  )\equiv \hat{\mu}(u_\epsilon  )$  and define $\bar{\rho} \in \mathcal{T} $ by: 
\begin{equation}
\bar{\rho}:=\begin{cases}
  t', \hspace{3 mm} \{\hat{\rho}<t'\}  \\
  \hat{\rho},\hspace{3 mm}\{\hat{\rho}\geq t'\}.
  \end{cases}
  \end{equation}
Hence, we have that:
\begin{align*}
J[t',x;u'_\epsilon ,\bar{\mu}(u'_\epsilon  )]&-J[t,x;u_\epsilon ,\hat{\mu}(u_\epsilon  )] =J[t',x;u'_\epsilon ,\bar{\rho}] - J[t,x;u_\epsilon ,\hat{\rho} ].
\end{align*}
From which we can now calculate that:
\begin{align*}
&\qquad\qquad\qquad\qquad\qquad\qquad J[t',x;u'_\epsilon ,\bar{\rho}] - J[t,x;u_\epsilon ,\hat{\rho} ]
\\&\begin{aligned}=-\mathbb{E}\Bigg[ \int_t^{\hat{\rho}\wedge\tau_S} &f(s,X_s^{t,x,u_\epsilon} )ds+\sum_{j\geq  1} c(\tau_j^\epsilon ,\xi_j^\epsilon  ) \cdot 1_{\{t< \tau_j^\epsilon <\hat{\rho}\wedge\tau_S \}} \\&
-\left(G\left(\bar{\rho}\wedge\tau_S,X_{\bar{\rho}\wedge\tau_S}^{t',x,u'_\epsilon } \right)-G\left(\hat{\rho}\wedge\tau_S,X_{\hat{\rho}\wedge\tau_S}^{t,x,u_\epsilon } \right)\right)\cdot 1_{\{\hat{\rho}\wedge \tau_s<\infty\}} \Bigg] \cdot 1_{\{\hat{\rho}\leq t' \}} 
\end{aligned}
\\&\begin{aligned}\quad+\mathbb{E}&\Bigg[ \int_{t'}^{\hat{\rho}\wedge\tau_S)}f(s,X_s^{t',x,u'_\epsilon } )ds- \int_t^{\hat{\rho}\wedge\tau_S} f(s,X_s^{t,x,u_\epsilon} )ds
\\&\qquad+\sum_{j\geq  1} \left(c\left(\tau_j^{'\epsilon} ,\xi_j^{'\epsilon}  \right) \cdot 1_{\{t'< \tau_j^{'\epsilon}<k\} }
 -   c(\tau_j^\epsilon ,\xi_j^\epsilon  ) \cdot 1_{\{t< \tau_j^\epsilon<k \}}\right)\cdot\delta^k_{\hat{\rho}\wedge\tau_S}
 \\&\qquad+\left(G\left(\bar{\rho}\wedge\tau_S,X_{\bar{\rho}\wedge\tau_S}^{t',x,u'_\epsilon } \right)-G\left(\hat{\rho}\wedge\tau_S,X_{\hat{\rho}\wedge\tau_S}^{t,x,u_\epsilon } \right)\right)\cdot 1_{\{\hat{\rho}\wedge \tau_s<\infty\}} \Bigg] \cdot 1_{\{\hat{\rho}> t' \}} 
\end{aligned} 
\\&\begin{aligned}
=\mathbb{E}\Bigg[ \int_{t'}^{\hat{\rho}\wedge\tau_S}&f(s,X_s^{t',x,u'_\epsilon } )ds- \int_t^{\hat{\rho}\wedge\tau_S} f(s,X_s^{t,x,u_\epsilon} )ds+c\left( t',\sum_{\tau_j^\epsilon \leq t'} \xi_j^\epsilon\right)-\sum_{\tau_j^\epsilon \leq t'} c\left(\tau_j^\epsilon ,\xi_j^\epsilon\right) 
\\&+\Big(G\left(\hat{\rho}\wedge\tau_S,X_{\hat{\rho}\wedge\tau_S}^{t',x,u'_\epsilon} \right)-G\left(\hat{\rho}\wedge\tau_S,X_{\hat{\rho}\wedge\tau_S}^{t,x,u_\epsilon }\right)\Big) \cdot 1_{\{\infty>\hat{\rho}> t' \}}\;\;\;&\nonumber
\\&+\Big(G\left(\hat{\rho}\wedge\tau_S,X_{\hat{\rho}\wedge\tau_S}^{t,x,u_\epsilon } \right)-G\left(\bar{\rho}\wedge\tau_S,X_{\bar{\rho}\wedge\tau_S}^{t',x,u'_\epsilon }\right)\Big) \cdot 1_{\{\hat{\rho}\leq t'<\infty \}}\Bigg]. 
\end{aligned}
\end{align*}
Now by assumption \hyperlink{A.3}{A.3}, we have that: 
\begin{equation}
\sum_{\tau_j^\epsilon \leq t'} c(\tau_j^\epsilon ,\xi_j^\epsilon  ) \geq c(t',\sum_{\tau_j^\epsilon \leq t'} \xi_j^\epsilon  ). 	\label{costdiscounteqn}\end{equation}
Hence, we find that: 
\begin{align}
&
J[t',x;u'_\epsilon ,\bar{\mu}(u'_\epsilon  )]-J[t,x;u_\epsilon ,\hat{\mu}(u_\epsilon  )]
\nonumber\\&
\begin{aligned}
\leq  \mathbb{E}\Bigg[\int_{t'}^{\hat{\rho}\wedge \tau_S} f(s,X_s^{t'  , x  , u'_\epsilon   } )ds &- \int_t^{\hat{\rho}\wedge \tau_S } f(s,X_s^{t, x  , u_\epsilon} )ds \Big
]
\\&+\sup_{\hat{\rho} \in  \{t',\hat{\rho}\}}\mathbb{E}\left[\left|G (\hat{\rho}\wedge \tau_S,X_{\hat{\rho}\wedge \tau_S}^{t, x  , u_\epsilon} )-G (\hat{\rho}\wedge \tau_S,X_{\hat{\rho}\wedge \tau_S}^{t', x  , u'_\epsilon} )\right|\right] ,
\label{JholdercontproofnoC} 
\end{aligned}
\end{align}
where we have used (\ref{costdiscounteqn}) to remove the cost terms. Now by the Lipschitz continuity of $G$ (c.f. \hyperlink{A.1.2.}{A.1.2.}) and Lemma \hyperlink{Lemma A.1.3}{A.1.3}, we can deduce the existence of a constant $c>0$ s.th. $\forall  s \in [0,T] $ we have that: 
\begin{equation}
\mathbb{E}\left[\left|G(s,X_s^{t',x,\cdot} )-G(s,X_s^{t,x,\cdot} )\right|\right]\leq c \sup_{s \in [0,T]} \mathbb{E}\left[\left|X_s^{t',x,\cdot}-X_s^{t,x,\cdot} \right|\right]\leq c\left|t-t' \right|^{\frac{1}{2}}, 
\label{Gholdercont}\end{equation}
where $c$ denotes some arbitrary constant (which may differ in each step of the proof).
Moreover, we can arrive at the result after observing a boundedness property of $f$ and invoking the Lipschitz property then appealing to the statements of Lemma \hyperlink{Lemma A.1.3}{A.1.3}. Indeed, using Fubini's Theorem and interchanging the integral and expectation, we have that: \begin{align}
\mathbb{E}\Bigg[ \int_t^{t'} f(s,X_s^{t,x,\cdot} )ds\Bigg]&\leq c  \mathbb{E} \Bigg[ \int_t^{t'}\sup_{s \in [t,t' ]} (1+|X_s^{t,x,\cdot} |)ds\Bigg]\nonumber\\
= c  \int_t^{t'}\sup_{s \in [t,t' ]} (1+\mathbb{E}[|X_s^{t,x,\cdot} |])ds&
\leq c(t'-t)(1+|x|) \in \mathbb{L}, \label{viscfbound}	
\end{align}
for some $c>0$  and where we have used the continuity of $f$.
Therefore, $\mathbb{E}[ \int_t^{t'} f(\cdot)ds] $ is bounded from above by $c(t'-t)$ for some $c>0$.
Using Fubini's Theorem, we now observe that we can deduce the existence of some constant $c>0$ s.th. 
\begin{align}
&\qquad\qquad\mathbb{E}\left[\int_{t'}^{\hat{\rho}\wedge \tau_S} f(s,X_s^{t'  , x  , \cdot   } )ds- \int_t^{\hat{\rho}\wedge \tau_S } f(s,X_s^{t, x  , \cdot} )ds \right
]
\nonumber
\\&\leq\mathbb{E}\left[\left| \int_t^{\rho'\wedge\tau_S} f(s,X_s^{t',x,\cdot} )-f(s,X_s^{t,x,\cdot} )ds+\int_t^{t'} f(s,X_s^{t,x,\cdot} )ds\right|  \right]\nonumber
\\&
\leq c\int_{t}^{\tau_S} \mathbb{E}|X_s^{t',x,\cdot}-X_s^{t,x,\cdot} |ds+c  \int_t^{t'}\sup_{s \in [t,t' ]} (1+\mathbb{E}[|X_s^{t,x,\cdot} |])ds 
\leq c|t-t' |^{\frac{1}{2}},	
\label{fpartlipscitz}
\end{align}
using the Lipschitzianity of $f$, (\ref{viscfbound}) and (ii) and (iii) of Lemma \hyperlink{Lemma A.1.3}{A.1.3}. 

Hence, after plugging (\ref{fpartlipscitz}) and (\ref{Gholdercont}) into (\ref{JholdercontproofnoC}) and (\ref{VplusholderproofJminusJ}), we can deduce that there exists a constant $c>0$ s.th. for $\epsilon >0$ the following estimate holds: 
\begin{align}
V^+  (t, x  )&- V^+  (t'  , x  )\leq |J[t',x;u'_\epsilon ,\bar{\mu}(u'_\epsilon  )]-J[t,x;u_\epsilon ,\hat{\mu}(u_\epsilon  )]|+3\epsilon \leq c|t-t' |^{\frac{1}{2}}+3\epsilon.	\label{vplusholderproofwithepsilon}
\end{align} 
Now, since $\epsilon$  in (\ref{vplusholderproofwithepsilon}) is arbitrary, we can deduce the existence of a constant $c>0$ s.th. 
\begin{equation}
|V^+ (t,x )-V^+ (t',x )|\leq c|t-t' |^{\frac{1}{2}},\quad \forall(t,x ),(t',x ) \in [0,T]\times S,
\end{equation}
after which we deduce (ii) holds for the function $V^+$.

To deduce that (ii) holds for the function $V^-$, we note that analogous to (\ref{vminusveqjminusj}), we have that: 
\begin{equation}
V^- (t',x)-V^- (t,x)= \hspace{-1 mm}\inf_{\alpha \in \mathcal{A}_{(0,T)}}  \sup_{\rho \in \mathcal{T}}  J[t',x;\alpha(\rho),\rho] - \inf_{\alpha \in \mathcal{A}_{(0,T)}}\sup_{\rho \in \mathcal{T}}  J[t,x;\alpha(\rho),\rho]. 	
\end{equation}
In a similar way to the proof of (ii) for $V^+$ we can deduce the existence of a constant $c>0$ s.th. 
\begin{equation}
|V^- (t,x )-V^- (t',x )|\leq c|t-t' |^{\frac{1}{2}}, 	
\end{equation}
from which we deduce the thesis.

\end{proof}
The following proposition establishes the boundedness of the value functions:
\begin{proposition}\label{Proposition 2.1.5}
The value functions $V^{\pm}$ are both bounded.
\end{proposition}
\begin{proof}
We do the proof for the function $V^-$ with the proof for $V^+$ being analogous.
Recall that: 
\begin{align}\nonumber
V^- (t_0,x_0)=\hspace{-3 mm} \inf_{\alpha \in \mathcal{A}_{(0,T)}}   \sup_{\rho \in \mathcal{T}}  \mathbb{E}\Bigg[ \int_{t_0}^{\rho\wedge\tau_S} f(s,X_s^{t_0,x_0,\alpha(\rho)} )ds&+\sum_{j\geq  1} c(\tau_j (\rho),\xi_j (\rho)) \cdot 1_{\{\tau_j (\rho)\leq \rho\wedge\tau_S \}}  
\\&
+G(\rho\wedge\tau_S,X_{\rho\wedge\tau_S}^{t_0,x_0,\alpha(\rho)} )1_{\{\rho\wedge\tau_S<\infty\}} \Bigg]. 	
\end{align}
Now let $u_0  \in \mathcal{U}$ be the player I control with which no impulses exercised. Then
\begin{equation}
X_{\tau_S }^{t_0  , x_0  ,u }\equiv X_{\tau_S}^{t_0, x_0  , u_0}+\sum_{j=1}^{\mu_{[t_0,T]} (u )} \xi_j  . 	
\end{equation}
If we denote by $Y_{\tau_S}^{t_0, x_0  ,u}:=X_{\tau_S}^{t_0, x_0  , u_0} + \sum_{j\geq  1}  \xi_j   \cdot 1_{\{\tau_j  < \tau_S  \wedge\rho \}}$ and using Lemma \hyperlink{Lemma A.1.3}{A.1.3} and by Gronwall's lemma we have that:
\begin{align*}
&\mathbb{E}\left[G (\tau_S,Y_{\tau_S}^{t_0, x_0  ,\alpha(\rho)} )-G (\tau_S,X_{\tau_S}^{t_0, x_0  , u_0} )\right]  
\\&\leq  c_2 \mathbb{E}[|Y_{\tau_S}^{t_0, x_0  , \alpha(\rho)}-X_{\tau_S}^{t_0, x_0  , u_0} | ]  
\\&\leq c_1 \mathbb{E}[((\rho\wedge\tau_S)-t_0 )^{\frac{1}{2}} )] ,
\end{align*} 
for some $c_1,c_2>0$. Moreover, since $u \in \mathcal{U}$ (hence  $\mathbb{E}[\mu_{[t_0,T]} (u)]< \infty$) we can find some $\lambda>0$ s.th. $\sum_{j\geq  1}  c (\tau_j  , \xi_j  ) \cdot 1_{\{\tau_j  < \tau_S  \wedge\rho \}}  \leq \lambda$ and hence: 
\begin{align*}
&\mathbb{E}\left[\sum_{j\geq  1}  c (\tau_j  , \xi_j  )  \cdot 1_{\{\tau_j  < \tau_S  \wedge\rho \}}  + G(\rho\wedge\tau_S,X_{\rho\wedge\tau_S}^{t_0,x_0,\alpha(\rho)}  )\right]
\\&\begin{aligned}= \mathbb{E}\Bigg[\sum_{j\geq  1}  c (\tau_j  , \xi_j  ) \cdot 1_{\{\tau_j  < \tau_S  \wedge\rho \}}   + \Big((G (\tau_S,Y_{\tau_S}^{t_0, x_0  ,\alpha(\rho)} )&-G(\tau_S,X_{\tau_S}^{t_0, x_0,u_0} )\Big)
\\&+G(\tau_S,X_{\tau_S}^{t_0, x_0,u_0} ))\Bigg] 
\end{aligned}
\\&\leq \mathbb{E}\left[G(\tau_S,X_{\tau_S}^{t_0,x_0,u_0} )+\lambda+c_1 ((\rho\wedge\tau_S)-t_0)^{\frac{1}{2}} \right].
\end{align*} 
Since by similar reasoning we can deduce that $ \mathbb{E}\left[\int_{t_0}^{\rho\wedge \tau_S} f(s,X_s^{t_0, x_0  ,u } )ds \right]\\\leq  \mathbb{E}[\int_{t_0}^{\rho\wedge \tau_S}f(s,X_s^{t_0, x_0  , u_0} )ds +c((\rho\wedge \tau_S  )- t_0  )^{\frac{1}{2}}  \cdot 1_{\{\mu_{[t_0, \tau_S]} (u )\}}  ]$ for some $c>0$; using the continuity of the functions $f$  and $G$  we find that: 
\begin{align*}
V^-  (t_0,x_0)&\leq   \sup_{\rho \in \mathcal{T}}   \mathbb{E}\Bigg[\int_{t_0}^{\rho\wedge \tau_S}f(s,X_s^{t_0, x_0  , u_0} )ds+G(\tau_S,X_{\tau_S}^{t_0, x_0,u_0} )\cdot 1_{\{{\rho}\wedge \tau_s<\infty\}}
\\&\qquad\qquad\qquad\qquad\qquad\;+ (\lambda+ c_1 |(\rho\wedge \tau_S  )- t_0  |^{\frac{1}{2}} ) \cdot 1_{\{\mu_{[t_0, \tau_S]} (u )\}}  \Bigg]
\\&\begin{aligned}
\leq    \sup_{\rho \in \mathcal{T}}  \mathbb{E}\Bigg[ \int_{t_0}^{\rho\wedge\tau_S} c_2 (1+\mathbb{E}[|X_s^{t_0,x_0,u_0} |])ds+c_3 (1+\mathbb{E}[|X_{\rho\wedge\tau_S}^{t_0,x_0,u_0} |])&
\\
+\left(\lambda+c_1 |(\rho\wedge\tau_S)-t_0 |^{\frac{1}{2}} \right) \cdot 1_{\{\mu_{[t_0,\tau_S]} (u)\}}  \Bigg] 	&
\end{aligned}
\\&\leq \sup_{\rho \in \mathcal{T}}  \mathbb{E}\left[\alpha+(\lambda+c_1 |(\rho\wedge\tau_S)-t_0|^{\frac{1}{2}} ) \cdot 1_{\{\mu_{[t_0,\tau_S]} (u)\}}  \right], 	
\end{align*} 
using Lemma \hyperlink{Lemma A.1.3}{A.1.3} and where $\alpha:= ((\rho\wedge \tau_S)-t_0)\cdot [c_2 + c_3](1+|x|)$  and $c_1>0$ and $c_2>0$ are constants. We then deduce the thesis since each of the terms inside the square bracket is bounded. 
\end{proof}

\begin{lemma}\label{Lemma 2.1.6.}
Let $V \in \mathcal{H}$ be a bounded function and $ (\tau,x) \in [0,T]\times S$ where $\tau$ is some $\mathcal{F}-$ measurable stopping time, then the set $\Xi(\tau,x) $ defined by: 
\begin{equation}
\Xi(\tau,x):=\left\{\xi \in \mathcal{Z}:\mathcal{M}V(\tau^-,x)=V(\tau,x+\xi)+c(\tau, \xi) \cdot 1_{\{\tau\leq T\}}  \right\}
\end{equation}
is non-empty.
\end{lemma}

The proof of the lemma is straightforward since we need only prove that the  infimum is in fact a minimum. This follows directly from the fact that the cost function is minimally bounded (c.f.  \hyperlink{AA.4.}{A.4.}) and that the value functions are also bounded by Proposition \ref{Proposition 2.1.5}.

A proof of the following lemma is reported in \cite{zhang2011stochastic}, Lemma 3.6 and similar result may be found in (Lemma 3.10 in \cite{chen2013impulse}):\

We give the following result without demonstration:
\

\begin{lemma}\label{Lemma 2.1.7.}
The non-local intervention operator $\mathcal{M}$ is continuous wherein we can deduce the existence of a constants $c_1,c_2>0$ s.th. when $s<s'$ with: 
\begin{enumerate}
    \item $\left|\mathcal{M}V^{\pm} (s,x)-\mathcal{M}V^{\pm} (s,y)\right|\leq c_1 |x-y|$,
\item  $\left|\mathcal{M}V^{\pm} (t,x)-\mathcal{M}V^{\pm} (s,x)]\right|\leq c_2 |t-s|^{\frac{1}{2}},\qquad \forall  (t,x),(s,y) \in [0,T]\times S 	$.
\end{enumerate}
\end{lemma}

\section{A Viscosity Theoretic Approach} \hypertarget{V}{}
 We now approach problem (\ref{viscp1obstacle}) using a viscosity theory approach. We shall firstly start by proving the existence of a value of the game and that the value is a unique viscosity solution to the HJBI equation, some ideas for the proofs in the section come from \cite{cosso2013stochastic}, \cite{cardaliaguet2010introduction,grun2013dynkin}. \\ 
The outline of the scheme is as follows:	 \renewcommand{\theenumi}{\roman{enumi}}
 \begin{enumerate}[leftmargin= 6.5 mm]
\item Using the DPP for each of the value functions, prove that the upper (resp., lower) value function is a viscosity subsolution (resp., supersolution) to the HJBI equation (\ref{viscp1obstacle}).
\item Prove a comparison theorem and the reverse inequality of (\ref{viscvaluefunctioninequalitystandard}) therefore proving equality of the value functions.
\item Using (ii), deduce the existence of a value of the game and that the value is a unique solution to the HJBI equation.\end{enumerate}

Let us therefore now introduce some key definitions relating to viscosity solutions:\

\begin{definition}\label{Definition 2.6.1}
A locally bounded lower (resp., upper) semicontinuous function $\psi:[0,T]\times S\to \mathbb{R}$ is a viscosity supersolution (resp., subsolution) to the HJBI equation (\ref{viscp1obstacle}) if:

For any $ (s,x) \in [0,T]\times S$ and $\psi \in \mathcal{C}^{1,2} ([0,T];S) $ s.th. $(s,x) $ is a local minimum (resp., maximum) of $V-\psi$, we have that: 
\begin{align}
\hspace{-1 mm}
\max\hspace{-0.2 mm}\left\{\min\left[-\frac{\partial \psi}{\partial s} (s,x)-\left(\mathcal{L}\psi(s,x)+f(s,x)\right),\psi(s,x)-G(s,x)\right]\hspace{-.5 mm},\psi(s,x)-\mathcal{M}\psi(s,x)\right\}\hspace{-.5 mm}\geq 0&\nonumber\\\label{viscsolutiondefn}
\text{  (resp.,}\leq 0\text{)}.& 
\end{align} 
A locally bounded lower (resp., upper) semicontinuous function $\psi:[0,T]\times S\to \mathbb{R}$ is a viscosity solution to the HJBI equation (\ref{viscp1obstacle}) if it is both a subsolution and supersolution of (\ref{viscsolutiondefn}).
\end{definition}

\begin{remark} If the value functions are known a priori to be continuous (or deterministic or in the discrete case) derivation of the DPP is straightforward. Otherwise, in general, we must use one of two arguments: a measurable selection argument or establish the regularity of the value functions then construct a measurable selection i.e. partition the state space then construct a measurable selection (this uses the Lindel\"{o}f property of the canonical space).
\end{remark}
We now state the dynamic programming principle for the game:
\begin{theorem}[Dynamic programming principle for stochastic differential games of control and stopping with Impulse Controls]\label{Dynamic_programming_principle_for_stochastic_differential_games_of_control_and_stopping_with_Impulse_Controls} 
Let $u \in \mathcal{U}$ be an admissible player I control and suppose $\rho \in \mathcal{T}$ is an $\mathcal{F}-$measurable stopping time, then for a sufficiently small $h$ the following variants of the DPP hold for the functions $V^+$ and $V^-$:
\begin{align}\nonumber
&V^- (t_0,x_0)
\\&\begin{aligned}= \inf_{\alpha \in \mathcal{A}_{(0,T)}}   \sup_{\rho \in \mathcal{T}}  \mathbb{E}\Bigg[\int_{t_0}^{(t_0+h)\wedge\rho} f(s,X_s^{t_0,x_0,\alpha(\rho)} )ds+\sum_{j\geq  1} c(\tau_j (\rho),\xi_j (\rho)) \cdot 1_{\{\tau_j (\rho)\leq (t_0+h)\wedge\rho\}}&
\\\nonumber\quad+G(\rho\wedge\tau_S,X_{\rho\wedge\tau_S}^{t_0,x_0,\alpha(\rho)} ) \cdot 1_{\{\rho\wedge\tau_S\leq t_0+h\}} +V^- ((t_0+h)\wedge\rho,X_{t_0+h}^{t_0,x_0,\alpha(\rho) }) \cdot 1_{\{\rho\wedge\tau_S>t_0+h\}}  \Bigg]& \label{viscDPPminus}
\end{aligned}
\end{align}
and 
\begin{align}\nonumber
&V^+(t_0,x_0)
\\&\begin{aligned}=\sup_{\beta \in \mathcal{B}_{(0,T)}}   \inf_{u \in \mathcal{U}}  \mathbb{E}\Bigg[\int_{t_0}^{(t_0+h)\wedge\beta(u)} f(s,X_s^{t_0,x_0,u} )ds+\sum_{j\geq  1} c(\tau_j,\xi_j ) \cdot 1_{\{\tau_j\leq (t_0+h)\wedge\beta(u)\}}&
\\\begin{aligned}&+G(\beta(u)\wedge\tau_S,X_{\beta(u)\wedge\tau_S}^{t_0,x_0,u} ) \cdot 1_{\{\beta(u)\wedge \tau_S\leq t_0+h\}}
\\&\begin{aligned}+V^- ((t_0+h)\wedge\beta(u),X_{t_0+h}^{t_0,x_0,u} ) \cdot 1_{\{\beta(u)\wedge\tau_S> t_0+h\}}  \Bigg]. \nonumber
\\\label{viscDPPplus}\forall (t_0,x_0 ) \in [0,T]\times S.&
\end{aligned}
\end{aligned}
\end{aligned}
\end{align}
\end{theorem}

Intuitively, the DPP states that if we compute the optimal controls on the intervals $[t_0,t_0+h]$ and $[t_0+h,t']$ for some $h<(t'\wedge\rho)-t_0$ for any $t'\in [0,T]$, then we would obtain the same result as that which we would obtain if we computed the optimal controls for the interval $[t_0,t']$ as a whole.

A classical consequence of the DPP (\ref{viscDPPminus}) and (\ref{viscDPPplus}) is that we find that the function $V^-$ (resp., $V^+$) is the subsolution (resp., supersolution) to an associated HJBI equation, namely (\ref{viscp1obstacle}). Moreover, if the game admits a value $V$ with $V \in \mathcal{C}^{1,2} ([0,T],\mathbb{R}^p )$, then the $V$ is a classical solution to an associated HJBI equation.
\begin{proof}

We begin by proving: 
\begin{align*}\nonumber 
&V^+ (t_0,x_0 )
\\&\begin{aligned}\nonumber\geq \sup_{\beta \in \mathcal{B}_{(0,T)} }   \inf_{u \in \mathcal{U}}  \mathbb{E}\Bigg[\int_{t_0}^{(t_0+h)\wedge\beta(u)\wedge\tau_S} f(s,X_s^{t_0,x_0,u} )ds+\sum_{j\geq  1} c(\tau_j,\xi_j ) \cdot 1_{\{\tau_j<t_0+h\wedge \beta(u)\wedge\tau_S \}}& 
\\\begin{aligned}&+G\left(\beta(u)\wedge\tau_S,X_{\beta(u)\wedge\tau_S}^{t_0,x_0,u}\right) \cdot 1_{\{\beta(u)\wedge\tau_S\leq t_0+h\}}  \Bigg]
\\&+V^+ ((t_0+h)\wedge\beta(u),X_{t_0+h}^{t_0,x_0,u} ) \cdot 1_{\{\beta(u)\wedge\tau_S> t_0+h\}},
\end{aligned}
\end{aligned}
\end{align*}
for some $\infty>h>0$.

Having established the uniform continuity of the functions $V^-$ and $V^+$, a countable selection argument is sufficient in order to derive the DPP, therefore avoiding measurable selection arguments directly. Indeed, using Proposition \ref{Proposition 2.1.4.}, we can find a set of controls that produce values of $J$ that are arbitrarily close to the values of $V^-$ and $V^+$ at some given point.

Hence, let $ (A^j )^{j \in \mathbb{N}}$ be a partition of $\mathbb{R}^p$. Let $\hat{\mu} \in \mathcal{B}_{(0,T)} $ be some $\epsilon -$optimal strategy against $\underset{\beta(u) \in \mathcal{B}_{(0,T)}}{\sup}   \underset{u \in \mathcal{U}}{\inf}V^+ (t_0+h\wedge\beta(u),X_{t_0+h}^{t_0,x_0,u} ) $. Note by Lemma \hyperlink{Lemma A.1}{A.1} we can deduce that since $\hat{\mu}$ is an $\epsilon -$optimal strategy against $\sup_{\beta(u) \in \mathcal{B}_{(0,T)}}   \underset{u \in \mathcal{U}}{\inf}V^+ (t_0+h\wedge\beta(u),X_{t_0+h}^{t_0,x_0,u} )$ then there exists some $2\epsilon -$optimal strategy $\hat{\mu}^x \in \mathcal{B}_{(0,T)} $ against $\sup_{\beta(u) \in \mathcal{B}_{(0,T)}}   \inf_{u \in \mathcal{U}}  V^+ (t_0+h\wedge\beta(u),x) $ s.th.\\ $\underset{\beta(u) \in \mathcal{B}_{(0,T)}}{\sup}   \underset{u \in \mathcal{U}}{\inf}  V^+ (t_0+h\wedge\beta(u),x)-(\epsilon +\delta )\leq  \underset{u \in \mathcal{U}}{\inf}  J[t_0+h,x;u_{[t_0+h,T ]} ,\hat{\mu}^x (u_{[t_0+h,T ]}  )] $ where $\epsilon >\delta >0$. Hence, we deduce that the strategy $\hat{\mu}^x$ is a $2\epsilon -$optimal strategy  $\inf_{u \in \mathcal{U}} V^+ (t_0+h\wedge\beta(u),y)$ for all $y \in B_\delta  (x) $ within some radius $0<\delta <\epsilon$.

Let us therefore construct the strategy $\mu$ by: 
\begin{equation}
\hat{\mu}(u)(s)=
\begin{cases}
\begin{aligned}
&\mu(u)(s),  & &s \in [t_0,t_0+h[ \\
  &\hat{\mu}^x(u_{[t_0+h,T ]})  (s)),& &s \in [t_0+h,T], X_{t_0+h}^{t_0,x_0,u} \in B_\delta  (x)
  \end{aligned}
  \end{cases}
  \end{equation}
Now for any $ (t_0,x_0 ) \in [0,T]\times S,u \in \mathcal{U},\mu \in \mathcal{B}_{(0,T)} $ and $\forall  u_{[t_0+h,T ]}  \in \mathcal{U}_{(t_0+h,T )} $ using Lemma \ref{Lemma 2.1.3.}, for some sufficiently small $h>0$, we have: 
\begin{align}\nonumber
&\begin{aligned}\quad\mathbb{E}\Bigg[\int_{t_0}^{\hat{\mu}(u)\wedge\tau_S} f(s,X_s^{t_0,x_0,u} )ds&+\sum_{j\geq  1} c(\tau'_j,\xi'_j ) \cdot 1_{\{\tau'_j\leq \hat{\mu}(u)\wedge\tau_S\}}  \\&
+G\left(\hat{\mu}(u)\wedge\tau_S,X_{\hat{\mu}(u)\wedge\tau_S}^{t_0,x_0,u}\right)\cdot 1_{\{\hat{\mu}(u)\wedge\tau_S<\infty\}} \Bigg] 	
 \nonumber\end{aligned}\\&
 \nonumber\begin{aligned}
 \geq \mathbb{E}\Bigg[\int_{t_0}^{(t_0+h)\wedge\mu(u)\wedge\tau_S} f(s,X_s^{t_0,x_0,u} )ds&+\sum_{j\geq  1} c(\tau_j,\xi_j ) \cdot 1_{\{\tau_j\leq (t_0+h)\wedge\mu(u)\wedge\tau_S\}}  
 \\&
 +G\left(\mu(u)\wedge\tau_S,X_{\mu(u)\wedge\tau_S}^{t_0,x_0,u}\right)  \cdot 1_{\{\mu(u)\wedge\tau_S<t_0+h\}}  \Bigg]
 \end{aligned}
 \\&\;\;
 \begin{aligned}+\mathbb{E}\Bigg[\int_{t_0+h}^{\hat{\mu}(u)\wedge\tau_S} &f(s,X_s^{t_0,x_0,u} )ds+\sum_{j\geq  1} c(\tau'_j,\xi'_j ) \cdot 1_{\{\hat{\mu}(u)\wedge\tau_S\geq \tau'_j> t_0+h\}} 
 \\&
 +G\left(\hat{\mu}(u)\wedge \tau_S,X_{\hat{\mu}(u)\wedge \tau_S}^{t_0,x_0,u}\right)\cdot 1_{\{\hat{\mu}(u)\wedge\tau_S<\infty\}} \Bigg] \cdot 1_{\{\hat{\mu}(u)\wedge\tau_S> t_0+h\}} -\epsilon,
 \label{viscproof32122ndline} 
\end{aligned}
\end{align}
for some arbitrary $\epsilon >0$. Using the properties of $X$, we can further rewrite (\ref{viscproof32122ndline}) as: 
\begin{align}&\nonumber\begin{aligned}
\mathbb{E}\Bigg[\int_{t_0}^{(t_0+h)\wedge\mu(u)\wedge\tau_S}f(s,X_s^{t_0,x_0,u} )ds&+\sum_{j\geq  1} c(\tau_j,\xi_j ) \cdot 1_{\{\tau_j\leq (t_0+h)\wedge\mu(u)\wedge\tau_S\}}
\\&+G\left(\mu(u)\wedge\tau_S,X_{\mu(u)\wedge\tau_S}^{t_0,x_0,u}\right)  \cdot 1_{\{\mu(u)\wedge\tau_S\leq t_0+h\}}\Bigg]
\end{aligned}
\\&
\begin{aligned}
+\mathbb{E}\Bigg[&\int_{t_0+h}^{\hat{\mu}(u_{[t_0+h,\tau_S ])}\wedge\tau_S}f(s,X_s^{t_0+h, X_{t_0+h}^{t_0,x_0,u},u_{[t_0+h,\tau_S ]} } )ds
\\&
+\sum_{j\geq  1} c(\tau'_j,\xi'_j ) \cdot 1_{\{\hat{\mu}(u_{[t_0+h,\tau_S ]} )\wedge\tau_S\geq \tau'_j> t_0+h}\} 
\\&\begin{aligned}\quad+ G\left(\hat{\mu}(u_{[t_0+h,\tau_S ]} )\wedge\tau_S,X_{\hat{\mu}(u_{[t_0+h,\tau_S ]} )\wedge\tau_S}^{t_0+h, X_{t_0+h}^{t_0,x_0,u} ,u_{[t_0+h,\tau_S ]}  } \right)\cdot& 1_{\{\hat{\mu}(u_{[t_0+h,\tau_S ]} )\wedge\tau_S<\infty\}} 
\\
\Bigg] \cdot& 1_{\{\hat{\mu}(u_{[t_0+h,\tau_S ]} )\wedge\tau_S\geq t_0+h}\} -\epsilon. 	\label{DPPproofp1step1}
\end{aligned}
\end{aligned}
\end{align}
We now exploit the regularity of $V$ (Proposition \ref{Proposition 2.1.4.}) and the $\epsilon-$optimality of $\mu$ by which we deduce that: 
\begin{align}&
\begin{aligned}
\mathbb{E}&\Bigg[\int_{t_0+h}^{\hat{\mu}(u_{[t_0+h,\tau_S ]})\wedge\tau_S} f(s,X_s^{t_0+h, X_{t_0+h}^{t_0,x_0,u},u_{[t_0+h,\tau_S ]}} )ds
\\&+\sum_{j\geq  1} c(\tau'_j,\xi'_j ) \cdot 1_{\{\hat{\mu}(u_{[t_0+h,\tau_S ]} )\wedge\tau_S\geq \tau'_j> t_0+h\}}  
\\&
+G\left(\hat{\mu}(u_{[t_0+h,\tau_S ]} )\wedge\tau_S,X_{\hat{\mu}(u_{[t_0+h,\tau_S ]} )\wedge\tau_S}^{t_0+h, X_{t_0+h} ^{t_0,x_0,u} ,u_{[t_0+h,\tau_S ]}  } \right)\cdot 1_{\{\hat{\mu}(u_{[t_0+h,\tau_S ]} )\wedge\tau_S\geq t_0+h}\} \Bigg]-\epsilon 	
\nonumber\end{aligned}
\\&\begin{aligned}
\geq \sum_{j \in \mathbb{N}} \mathbb{E}\Bigg[&\int_{t_0+h}^{\hat{\mu}(u)\wedge\tau_S}f(s,X_s^{t_0+h,y^j,u_{[t_0+h,\tau_S ]} } )ds+\sum_{j\geq  1} c(\tau'_j,\xi'_j ) \cdot 1_{\{\hat{\mu}  (u)\wedge\tau_S\geq \tau'_j\geq t_0+h\}} 
\\&
+G\left(\hat{\mu}(u)\wedge\tau_S,X_{\hat{\mu}(u)\wedge\tau_S}^{t_0+h,y^j,u_{[t_0+h,\tau_S ]}  } \right)\cdot 1_{\{\hat{\mu}(u_{[t_0+h,\tau_S ]} )\wedge\tau_S\geq t_0+h}\} \Bigg]-c\delta  -\epsilon    	\nonumber\end{aligned}
\\&\qquad\qquad\qquad\qquad\qquad\geq V^+ (t_0+h\wedge\hat{\mu}(u),X_{t_0+h}^{t_0,x_0,u} )-2\epsilon -c\delta ,	\label{DPPproofp1epsilondelta}
\end{align}
using the $\epsilon -$optimality of the strategy $\hat{\mu}$ against \\$\sup_{\beta \in \mathcal{B}_{(0,T)}}    \inf_{u \in \mathcal{U}}V^+ (t_0+h\wedge\beta(u),X_{t_0+h}^{t_0,x_0,u} )$.
Hence, putting (\ref{DPPproofp1epsilondelta}) together with (\ref{DPPproofp1step1}), we find that: 
\begin{align*}
&\begin{aligned} 
 \mathbb{E}\Bigg[\int_{t_0}^{(t_0+h)\wedge\mu(u)\wedge\tau_S} f(s,X_s^{t_0,x_0,u} )ds&+\sum_{j\geq  1} c(\tau_j,\xi_j ) \cdot 1_{\{\tau_j\leq (t_0+h)\wedge\mu(u)\}} \\&
+G(\mu(u)\wedge\tau_S,X_{\mu(u)}^{t_0,x_0,u} )  \cdot 1_{\{\mu(u)\wedge\tau_S\leq t_0+h\}}  ]
\end{aligned}
\\&\begin{aligned}+\mathbb{E}\Bigg[\int_{t_0+h}^{\mu(u)\wedge\tau_S} f(s,X_s^{t_0,x_0,u} )ds&+\sum_{j\geq  1} c(\tau_j,\xi_j ) \cdot 1_{\{\mu(u)\wedge\tau_S\geq \tau_j> t_0+h\}}  \\&+G(\mu(u)\wedge\tau_S,X_{\mu(u)\wedge\tau_S}^{t_0,x_0,u} )\cdot 1_{\{\mu(u)\wedge\tau_S<\infty\}} \Bigg] \cdot 1_{\{\mu(u)\geq t_0+h\}}
\end{aligned}
\\& \begin{aligned}
\geq \mathbb{E}\Bigg[\int_{t_0}^{(t_0+h)\wedge\mu(u)\wedge\tau_S} f(s,X_s^{t_0,x_0,u} )ds+\sum_{j\geq  1} c(\tau_j,\xi_j ) \cdot 1_{\{\tau_j\leq (t_0+h)\wedge\mu(u)\wedge\tau_S\}} &
\\ +G(\mu(u)\wedge\tau_S,X_{\mu(u)\wedge\tau_S}^{t_0,x_0,u} ) \cdot 1_{\{\mu(u)\wedge\tau_S\leq t_0+h\}}  &
\\+V^+ (t_0+h\wedge\hat{\mu}(u),X_{t_0+h}^{t_0,x_0 ,u} ) \cdot 1_{\{\hat{\mu}(u)\wedge\tau_S> t_0+h\}}  \Bigg]-c\delta -2\epsilon, &
\end{aligned} 
\end{align*}
from which after successively applying the  $\inf$ and $\sup$ operators we deduce the first result since $\delta$  and $\epsilon$  can be chosen arbitrarily.

We prove the reverse inequality in an analogous manner, in particular, we now prove inequality for the function $V^-$. Indeed, by Proposition \ref{Proposition 2.1.4.}, we can deduce the existence of a strategy $\alpha^{(1,\epsilon ) } \in \mathcal{A}_{(0,T)} $ against $V^- (t,x) $ s.th. 
\begin{align}
&\begin{aligned}
\inf_{\alpha \in \mathcal{A}_{(0,T)}}  \sup_{\rho \in \mathcal{T}} \mathbb{E}&\Bigg[\int_{t_0}^{(t_0+h)\wedge\rho\wedge\tau_S} f(s,X_s^{t_0,x_0,\alpha(\rho)} )ds+\sum_{j\geq  1} c(\tau_j,\xi_j ) \cdot 1_{\{\tau_j\leq (t_0+h)\wedge\rho\wedge\tau_S\}}  \\&\begin{aligned}+G(\rho\wedge\tau_S,X_{\rho\wedge\tau_S}^{t_0,x_0,\alpha(\rho)} ) \cdot 1_{\{\rho\wedge\tau_S\leq t_0+h\}}+V^- (t_0+h,X_{t_0+h}^{t_0,x_0,\alpha(\rho)} )&
\\\cdot 1_{\{\rho> t_0+h\}}  \Bigg]&
\end{aligned} 	
\end{aligned} 	
\nonumber\\&
\begin{aligned}
\geq &\mathbb{E}\Bigg[\int_{t_0}^{(t_0+h)\wedge\rho\wedge\tau_S} f(s,X_s^{t_0,x_0,\alpha^{(1,\epsilon ) } (\rho)} )ds+\sum_{j\geq  1} c(\tau_j^{1,\epsilon },\xi_j^{1,\epsilon } ) \cdot 1_{\{\tau_j^{1,\epsilon }\leq (t_0+h)\wedge\rho\wedge\tau_S\}}  \\&\begin{aligned}\;+G(\rho\wedge\tau_S,X_{\rho\wedge\tau_S}^{t_0,x_0,\alpha^{(1,\epsilon ) } (\rho) ) } \cdot 1_{\{\rho\wedge\tau_S\leq t_0+h\}}  +V^- (t_0+h,X_{t_0+h}^{t_0,x_0,\alpha^{(1,\epsilon ) } (\rho) } \cdot 1_{\{\rho\wedge\tau_S> t_0+h\}}  \Bigg]&
\\-\epsilon& , 	\label{DPPproofp2step1}
\end{aligned}
\end{aligned}
\end{align}
where we have used the shorthand $\tau_j (\rho)\equiv \tau_j$  and $\xi_j (\rho)\equiv \xi_j$, $\forall  j \in \mathbb{N}$. 

We now build the strategy $\alpha$ by: \begin{equation}
\alpha(\rho)(s)=
\begin{cases}
\begin{aligned}
 &\alpha^{(1,\epsilon ) } (\rho)(s),&  &s \in [t_0,t_0+h[ \\&
  \alpha^{A_i } (\rho' )(s),& &s \in [t_0+h,T ],X_{t_0+h}^{t_0,x_0,\alpha^{(1,\epsilon ) }} \in A_i,
  \end{aligned}
  \end{cases}\nonumber \end{equation}
where we have used $\rho'$ to denote the player II stopping time s.th. $\rho' \in \mathcal{T}_{(t_0+h,T)}\subset  [t_0+h,T ] $.
Let $\alpha^{(2,\epsilon )} \in \mathcal{A}_{(0,T)} $ be an $\epsilon -$optimal strategy against $\sup_{\rho \in \mathcal{T}}V^- ((t_0+h)\wedge\rho,x) $. Using Lemma \ref{Lemma 2.1.3.} and by similar reasoning as in part (i), we can also deduce the existence of a strategy $\alpha^{A_i } \in \mathcal{A}_{(t_0+h,T)} $ s.th. $\forall  y \in A_i,\rho' \in \mathcal{T}_{(t_0+h,T)} $ and some $\epsilon >0$, the following inequality holds: \begin{equation}
V^- ((t_0+h)\wedge\rho,y)\geq  J[(t_0+h)\wedge\rho,y;\alpha^{\mathcal{A}_i } (\rho' )]-\epsilon . 	\end{equation}
We therefore observe that: 
\begin{align}
&\mathbb{E}\left[V^-\left(t_0+h,X_{t_0+h}^{t_0,x_0,\alpha^{(1,\epsilon ) } (\rho')} \right)\right]\nonumber
\\&=\mathbb{E}\left[\sum_{i\geq 1} V^- \left(t_0+h,X_{t_0+h}^{t_0,x_0,\alpha^{(1,\epsilon ) } (\rho')} \right) \cdot 1_{\{X_{t_0+h}^{t_0,x_0,\alpha^{(1,\epsilon ) } (\rho'))} \in \mathcal{A}_i \}}  \right] 	
\nonumber
\\&\geq \mathbb{E}\left[\sum_{i\geq 1} J\left[t_0+h,X_{t_0+h}^{t_0,x_0,\alpha^{(1,\epsilon )} (\rho')};\alpha^{\mathcal{A}_i } (\rho' ),\rho'\right] \cdot 1_{\{X_{t_0+h}^{t_0,x_0,\alpha^{(1,\epsilon ) } (\rho' )} \in \mathcal{A}_i \}}  \right]-\epsilon  	
\nonumber
\\&= J\left[t_0+h,X_{t_0+h}^{t_0,x_0,\alpha^{(1,\epsilon )} (\rho')};\sum_{i\geq 1} \alpha^{A_i} (\rho' ) \cdot 1_{\{X_{t_0+h}^{t_0,x_0,\alpha^{(1,\epsilon ) } (\rho' )} \in \mathcal{A}_i \}} ,\rho'\right]-\epsilon . 	\label{DPPproofp2Vopenballineq}
\end{align}
 Let us now construct the strategy $\bar{\alpha}^{(2,\epsilon ) } (\rho) \in \bar{\mathcal{A}}(t_0+h) $ defined by:\\$
\bar{\alpha}^{(2,\epsilon ) } (\rho):=\sum_{i\geq 1} \alpha^{\mathcal{A}_i } (\rho' ) \cdot 1_{\{X_{t_0+h}^{t_0,x_0,\alpha^{(1,\epsilon ) } (\rho' )}  \in \mathcal{A}_i \}} 
$.
Now, after introducing the strategy $\bar{\alpha}^{(2,\epsilon ) }$ to (\ref{DPPproofp2Vopenballineq}), we deduce that: 
\begin{align}
&\begin{aligned}
\mathbb{E}\Bigg[&\int_{t_0}^{(t_0+h)\wedge\rho\wedge\tau_S} f(s,X_s^{t_0,x_0,\alpha^{(1,\epsilon ) } (\rho)} )ds+\sum_{j\geq  1} c(\tau_j^{1,\epsilon },\xi_j^{1,\epsilon } ) \cdot 1_{\{\tau_j^{1,\epsilon }\leq (t_0+h)\wedge\rho\wedge\tau_S\}}\\&
\begin{aligned} +G\left(\rho\wedge\tau_S,X_{\rho\wedge\tau_S}^{t_0,x_0,\alpha^{(1,\epsilon ) } (\rho)}\right) \cdot 1_{\{\rho\wedge\tau_S\leq t_0+h\}}  +V^- \left(t_0+h,X_{t_0+h}^{t_0,x_0,\alpha^{(1,\epsilon ) } (\rho)} \right)&
\\
\cdot 1_{\{\rho\wedge\tau_S> t_0+h\}}  \Bigg]&-\epsilon
\end{aligned}
\end{aligned}
\nonumber
\\&\begin{aligned}\quad\geq& \mathbb{E}\Bigg[ \int_{t_0}^{(t_0+h)\wedge\rho\wedge\tau_S} f(s,X_s^{t_0,x_0,\alpha^{(1,\epsilon ) } (\rho)}  )ds+\sum_{j\geq  1} c(\tau_j^{1,\epsilon },\xi_j^{1,\epsilon } ) \cdot 1_{\{\tau_j^{1,\epsilon }\leq (t_0+h)\wedge\rho\wedge\tau_S \}}  
\\&\begin{aligned}+G\left(\rho\wedge\tau_S,X_{\rho\wedge\tau_S}^{t_0,x_0,\alpha^{(1,\epsilon ) } (\rho)}\right) \cdot 1_{\{\rho\wedge\tau_S\leq t_0+h\}}  +J\left[t_0+h,X_{t_0+h}^{t_0,x_0,\alpha^{(1,\epsilon )} (\rho')};\bar{\alpha}^{(2,\epsilon ) } (\rho' ),\rho'\right]&
\\\cdot 1_{\{\rho\wedge\tau_S> t_0+h\}}  \Bigg]-2\epsilon.& \label{DPPproofp22epsilon}	
\end{aligned}
\end{aligned}
\end{align}
 We lastly construct the strategy $\alpha^\epsilon  \in \mathcal{A}(t_0) $ which consists of the strategy $\alpha^{(1,\epsilon ) }$ which is played up to time $t_0+h$ at which point the strategy $\bar{\alpha}^{(2,\epsilon ) }$ is then played.

Hence, after putting (\ref{DPPproofp22epsilon}) and (\ref{DPPproofp2step1}) together we observe that: 
\begin{align}
&
V^- (t_0,x_0)
\\&\begin{aligned}\geq \mathbb{E}\Bigg[&\int_{t_0}^{(t_0+h)\wedge\rho\wedge\tau_S} f(s,X_s^{t_0,x_0,\alpha^{(1,\epsilon ) } (\rho)} )ds+\sum_{j\geq  1} c(\tau_j^{1,\epsilon },\xi_j^{1,\epsilon } ) \cdot 1_{\{\tau_j^{1,\epsilon }\leq (t_0+h)\wedge\rho\wedge\tau_S\}}  
\\&\begin{aligned}+G(\rho,X_\rho^{t_0,x_0,\alpha^{(1,\epsilon ) } (\rho) ) } \cdot 1_{\{\rho\wedge\tau_S\leq t_0+h\}} +V^- (t_0+h,X_{t_0+h}^{t_0,x_0,\alpha^{(1,\epsilon ) } (\rho) })&
\\\cdot 1_{\{\rho\wedge\tau_S> t_0+h\}}  \Bigg]&-\epsilon
\end{aligned}
\end{aligned}
\nonumber\\&
\begin{aligned}
\geq\mathbb{E}\Bigg[&\int_{t_0}^{(t_0+h)\wedge\rho\wedge\tau_S} f(s,X_s^{t_0,x_0,\alpha^\epsilon  (\rho)} )ds+\sum_{j\geq  1} c(\tau_j^\epsilon ,\xi_j^\epsilon  ) \cdot 1_{\{\tau_j^\epsilon \leq (t_0+h)\wedge\rho\wedge\tau_S\}}  \\&
\begin{aligned}
+G(\rho\wedge\tau_S,X_{\rho\wedge\tau_S}^{t_0,x_0,\alpha^\epsilon  (\rho) } ) \cdot 1_{\{\rho\wedge\tau_S\leq t_0+h\}} +J\left[t_0+h,X_{t_0+h}^{t_0,x_0,\alpha^\epsilon  (\rho')};\alpha^\epsilon  (\rho' ),\rho'\right]&
\\\cdot 1_{\{\rho\wedge\tau_S> t_0+h\}}  \Bigg]&-2\epsilon . 
\end{aligned}
\end{aligned}
\end{align}

Moreover, since $\epsilon$ is arbitrary, we readily deduce that:  
\begin{align*}\nonumber
&V^- (t_0,x_0 )
\\&\begin{aligned}\geq  &\inf_{\alpha \in \mathcal{A}_{(0,T)}}  \sup_{\rho \in \mathcal{T}}  \mathbb{E}\Bigg[\int_{t_0}^{(t_0+h)\wedge\rho\wedge\tau_S} f(s,X_s^{t_0,x_0,\alpha(\rho)} )ds+\sum_{j\geq  1} c(\tau_j,\xi_j ) \cdot 1_{\{\tau_j\leq (t_0+h)\wedge\rho\wedge\tau_S\}}  \\&+G(\rho\wedge\tau_S,X_{\rho\wedge\tau_S}^{t_0,x_0,\alpha(\rho) }) \cdot 1_{\{\rho\wedge\tau_S\leq t_0+h\}}  +V^- (t_0+h,X_{t_0+h}^{t_0,x_0,\alpha(\rho) }) \cdot 1_{\{\rho\wedge\tau_S> t_0+h\}}  \Bigg], 	
\end{aligned}
\end{align*}
from which we readily deduce the required result. We can prove the reverse inequality in an analogous manner for which, in conjunction with (\ref{viscDPPplus}) proves the thesis.

\end{proof}

A classical consequence of the dynamic programming principles (\ref{viscDPPminus}) and (\ref{viscDPPplus}) is that we find that the function $V^-$ (resp., $V^+$) is the subsolution (resp., supersolution) to an associated HJBI equation, namely (\ref{viscp1obstacle}). Moreover, if the game admits a value $V$, s.th. $V \in \mathcal{C}^{1,2} ([0, T][0, T],\mathbb{R}^p )$, then the $V$ is a classical solution to an associated HJBI equation. 

The following lemma characterises the conditions in which the value of the game satisfies a HJBI equation: 

\begin{lemma}\label{Lemma 2.6.2.}
The function $V^- $ is a viscosity supersolution of (\ref{viscp1obstacle}) and the $V^+$ is a viscosity subsolution of (\ref{viscp1obstacle}).
\end{lemma}

\begin{proof}
The lemma is proven by contradiction.

We begin by proving that $V^+$ is a viscosity subsolution of (\ref{viscp1obstacle}).

Suppose $\psi:[0,T]\times S\to \mathbb{R}$ is a test function with $\psi \in \mathcal{C}^{1,2} ([0,T],S) $ and $ (t,x) \in [0,T]\times S$ are s.th. $V^+-\psi$ attains a local minimum at $(t,x) $ with $V^+(t,x)-\psi(t,x)=0$. We note that it remains only to show that $\forall (s,x) \in [0,T]\times S, \frac{\partial \psi}{\partial s} (s,x)+\mathcal{L}\psi(s,x)+f (s,x )\geq 0$ whenever $\psi(s,x)-G(s,x)>0$ which, as we now show, follows as a classical consequence of the DPP.

Indeed, by Proposition \ref{Proposition 2.1.4.} we can deduce the existence of a $\epsilon -$optimal strategy $\alpha^\epsilon  \in \mathcal{A}_{(0,T)} $ to which the associated control is $\alpha^\epsilon  (v)\equiv u^\epsilon  \in \mathcal{U}$ (against $V^+(t_0,x_0)$) s.th. \begin{align}
&\qquad\qquad\qquad\qquad\qquad\qquad\qquad\qquad\psi (t_0,x_0)=V^+(t_0,x_0)\nonumber
\\&\begin{aligned}
\geq  \inf_{u \in \mathcal{U}}\mathbb{E}\Bigg[\int_{t_0}^{\tau_S\wedge\rho}&f (s,X_s^{t_0,x_0,u} )ds+\sum_{j\geq  1} c(\tau_j,\xi_j) \cdot 1_{\{\tau_j\leq\tau_S\wedge\rho\}} 
\\&+G(\tau_S\wedge\rho,X_{\tau_S}^{t_0,x_0,u} )\cdot 1_{\{\tau_S\wedge\rho<t_0+h \}} 
+V^+ (t_0+h,X_{t_0+h}^{t_0,x_0,u} ) \cdot 1_{\{t_0+h=\rho \}}  \Bigg] 	\end{aligned}
\nonumber\\&
\begin{aligned}\geq \mathbb{E}\Bigg[\int_{t_0}^{\tau_S\wedge\rho}&f (s,X_s^{t_0,x_0,u_\epsilon } )ds+\sum_{j\geq  1} c(\tau_j^\epsilon ,\xi_j^\epsilon  )  \cdot 1_{\{\tau_j\leq\tau_S\wedge\rho\}} \\&\begin{aligned}+G(\tau_S\wedge\rho,X_{\tau_S\wedge\rho}^{t_0,x_0,u_\epsilon   } ) \cdot 1_{\{\tau_S\wedge\rho<t_0+h \}} 
+V^+ (t_0+h,X_{t_0+h}^{t_0,x_0,u_\epsilon} ) \cdot 1_{\{t_0+h=\rho \}}  \Bigg]&
\\-\epsilon h.&
\label{viscviscproofstep1}
\end{aligned}
\end{aligned}
\end{align}
Let us now define: 
\begin{align}\nonumber
\phi^{[h]}  (t_0,x_0 ):=  \mathbb{E}\Bigg[\int_{t_0}^{\tau_s  \wedge\rho} f (s,X_s^{t_0, x_0,u_0} )ds&+G (\tau_S  \wedge\rho ,X_{\tau_S  \wedge\rho }^{t_0, x_0  , u_0} ) \cdot 1_{\{\tau_S  \wedge\rho< t_0  +h  \}}
\\&+\psi (t_0+h,X_{t  +h }^{t_0, x_0  , u_0} ) \cdot 1_{\{t_0  +h=\rho  \}}  \Bigg], 	\label{viscviscproofphidefn}
\end{align}
where $u_0  \in \mathcal{U}$ is the player I control s.th. no impulses are exercised.
We firstly wish to show that given $\epsilon >0$ we have that: \begin{equation}
\psi (t_0,x_0)\geq  \phi^{[h]}  (t_0,x_0 )-2\epsilon h . 	\label{viscviscproofphipsiineq}
\end{equation}
 Indeed, we firstly note that:
\begin{equation}
X_{t_0+h}^{t_0,x_0,u_\epsilon}\equiv X_{t_0+h}^{t_0,x_0,u_{0_{[t_0,t_0+h]}}  }+\sum_{j=1}^{\mu_{[t_0,t_0+h]} (u_\epsilon  )} \xi_j^\epsilon    . 	
\end{equation}
We now exploit the regularity of $V$ and the boundedness of the sequence of intervention costs, indeed we have that: 
\begin{align}
\mathbb{E}&\Bigg[\sum_{j\geq  1} c(\tau_j^\epsilon ,\xi_j^\epsilon  )  \cdot 1_{\{\tau_j^\epsilon\leq (t_0+h)\wedge\rho\}} +V^+ (t_0+h,X_{t_0+h}^{t_0,x_0,u_\epsilon} ) \cdot 1_{\{t_0+h=\rho \}}  \Bigg] 
\nonumber\\\nonumber\geq \mathbb{E}&\Bigg[(V^+ (t_0+h,X_{t_0+h}^{t_0,x_0} )+(\lambda-c(\rho-t_0 )^{\beta} ) \cdot 1_{\mu_{[t_0,t_0+h] (u_\epsilon   )}\geq 1  } \cdot 1_{\{t_0+h=\rho \}}  \Bigg] 	
\\\geq \mathbb{E}&[(\psi (t_0  +h, X_{t_0+h}^{t_0,x_0} )+(\lambda-c(\rho-t_0 )^{\beta} ) \cdot 1_{\mu_{[t_0,t_0+h] (u_\epsilon   )}\geq 1} \cdot 1_{\{t_0  +h=\rho  \}}  ], 	\label{viscviscproofimpulseremovalestimatepsi}
\end{align}
where we have used the fact that $ \sum_{j\geq 1}\inf_{z\in\mathcal{Z}}c(\tau_j^\epsilon ,z) \geq  \lambda\cdot 1_{\mu_{[t_0,t_0+h]} (u_\epsilon   )} $  for some $\lambda>0$ and that if $X'\equiv  X_{t_0+h}^{t_0,x_0,u_{0_{[t_0,t_0+h]}}  }+\sum_{j=1}^{\mu_{[t_0,t_0+h]} (u_\epsilon )} \xi_j^\epsilon$, we have that $V^+ (t_0+h,X_{t_0+h}^{t_0,x_0} )=V^+ (t_0+h,X' )+(V^+ (t_0+h,X_{t_0+h}^{t_0,x_0} )-V^+ (t_0+h,X' ))\leq V^+ (t_0+h,X' )+ch^{\frac{   1}{2}}$ for some $c>0$ using Lemma \hyperlink{Ap.4}{A.3.} and Gronwall's lemma.
Using the same arguments we can similarly deduce that there exists some constant $c>0$ such that:
\begin{gather}
G (t_0+h,X_{t_0+h}^{t_0,x_0,u_\epsilon})+f (t_0+h,X_{t_0+h}^{t_0,x_0,u_\epsilon})\geq f (t_0+h,X_{t_0+h}^{t_0,x_0} ) + G (t_0+h,X_{t_0+h}^{t_0,x_0} )-ch^{\frac{1}{2}}.
\label{viscviscproofGFfreeestimates}\end{gather}
 Now, since $ (\lambda-c(\rho-t_0 )^{\frac{1}{2}} ) \cdot 1_{\mu_{[t_0,t_0+h]} (u_\epsilon   )\geq 1} \cdot 1_{\{t_0  +h=\rho  \}} =(\lambda-ch^{\frac{1}{2}} ) \cdot 1_{\mu_{[t_0,t_0+h]} (u_\epsilon   )\geq 1} $  and since there exists $\bar{s} \in ]t_0,T[ $ s.th. for $h \in [t_0,\bar{s}]$ for any $\epsilon >0$ we have that: \begin{equation}
 (\lambda-ch^{\frac{1}{2}}  ) \cdot 1_{\mu_{[t_0,t_0+h]} (u_\epsilon  )\geq 1} \geq -\epsilon h,  	\label{viscviscproofconstantestimates}\end{equation}
we observe that after inserting (\ref{viscviscproofconstantestimates}) and (\ref{viscviscproofGFfreeestimates}) into (\ref{viscviscproofimpulseremovalestimatepsi}) and (\ref{viscviscproofstep1}), we deduce that (\ref{viscviscproofphipsiineq}) does indeed hold.

Hence, combining (\ref{viscviscproofphidefn}) and (\ref{viscviscproofphipsiineq}) we find that: \begin{align}
&\qquad\qquad\qquad\qquad\qquad\qquad\qquad\qquad\psi (t_0,x_0)=V^+(t_0,x_0)
\nonumber\\&
\begin{aligned}
\geq  \inf_{u \in \mathcal{U}}\mathbb{E}\Bigg[\int_{t_0}^{\tau_S\wedge\rho}&f (s,X_s^{t_0,x_0,u} )ds+\sum_{j\geq  1} c(\tau_j,\xi_j) \cdot 1_{\{\tau_j\leq\tau_S\wedge\rho\}}  
\\&+G(\tau_S  \wedge\rho,X_{\tau_S}^{t_0,x_0,u} )\cdot 1_{\{\tau_S  \wedge\rho< t_0  +h  \}}+V^+ (t_0+h,X_{t_0+h}^{t_0,x_0,u} ) \cdot 1_{\{t_0+h=\rho  \}}  \Bigg] \end{aligned}
\nonumber\\
&\begin{aligned}
\geq  \mathbb{E}\Bigg[\int_{t_
0}^{\tau_S  \wedge\rho} f (s, X_s^{t_0  , x_0  , u_\epsilon} )ds &+G (\tau_S  \wedge\rho ,X_{\tau_S  \wedge\rho }^{t_0  , x_0  , u_\epsilon} )\cdot 1_{\{\tau_S  \wedge\rho< t_0  +h  \}}  
\\&+\psi (t_0  +h, X_{t_0  +h}^{t_0  , x_0} ) \cdot 1_{\{t_0  +h=\rho  \}}  \Bigg]-2\epsilon h.   	\label{viscviscproofpsiweakdpp}
\end{aligned}
\end{align}
Let us now define as $\Lambda (s,x):=(\frac{\partial}{\partial s}+\mathcal{L})\psi (s,x) $. By It\={o}'s formula for c\`{a}dl\`{a}g semi-martingale processes (see for example Theorem II.33 of \cite{protter2005stochastic}), we have that: 
\begin{align}\nonumber
\psi(t_0,x_0)=\psi(t_0+h,X_{t_0+h}^{t_0,x_0,u_\epsilon} )-\int_{t_0}^{t_0+h\wedge\tau_S\wedge\rho}  \left\langle \nabla_x\psi(s,X_s^{t_0,x_0,u_\epsilon } ),\mu(s,X_s^{t_0,x_0,u_\epsilon } )\right\rangle dB_s&
\\-\int_{t_0}^{t_0+h\wedge\tau_S\wedge\rho}  \Lambda (s,X_s^{t_0,x_0,u_\epsilon } )ds.& 	 \end{align}

In order to generate a contradiction, we assume that $G(s,x)-V^+(s,x)=G(s,x)-\psi(s,x)\geq 0$ and suppose that the supposition of the lemma is false so that $\Lambda (s,x)+f (s,x )>0$. We can therefore consider constants $a,h,
\delta >0$ s.th. $\forall  (s,x) \in [t_0,t_0+h]\times B_a (x) $ s.th. $G(s,x)-\psi(s,x)\geq \delta$  and $\Lambda (s,x)+f (s,x )\geq \delta$ .

Let us now define the set $E:=\{ \inf_{s \in [t_0,t_0+h]}| X_s^{t_0,x_0,\cdot}-x|>a\}$ then using Lemma \hyperlink{Lemma A.1.3}{A.1.3} (i.e. the $\frac{1}{2}$-H\"{o}lder continuity of $X$) and by Tchebyshev's inequality, we can deduce the existence of a constant $c>0$ that depends only on the parameters of $X_s^{t_0,x_0,\cdot}$ s.th. $\mathbb{P}[E]\leq  \inf_{s \in [t_0,t_0+h]}  \frac{c(s-t_0 )^2}{a^4} \leq \frac{c h^2 }{a^4}$  . 

Then since 
\begin{equation}
\mathbb{E}\left[\psi(t_0+h,X_{t_0+h}^{t_0,x_0,u_\epsilon} )\right]-\psi(t_0,x_0)=\mathbb{E}\left[\int_{t_0}^{(t_0+h)\wedge\tau_S}  \Lambda (s,X_s^{t_0,x_0,u_\epsilon } )ds\right], 	
\end{equation}
we have that: 
\begin{align*}
&\begin{aligned}
-\psi(t_0,x_0)=\mathbb{E}&\left[ 1_{E^c}\cdot\left(\int_{t_0}^{(t_0+h)\wedge\tau_S} \Lambda (s,X_s^{t_0,x_0,u_\epsilon } )ds-\psi(t_0+h,X_{t_0+h}^{t_0,x_0,u_\epsilon} )\right)\right]\\&\hfill+\mathbb{E}\left[1_E\cdot\left (\int_{t_0}^{(t_0+h)\wedge\tau_S}  \Lambda (s,X_s^{t_0,x_0,u_\epsilon } )ds-\psi(t_0+h,X_{t_0+h}^{t_0,x_0,u_\epsilon} )\right)\right]
\end{aligned}
\\&\qquad\quad
\geq \mathbb{E}\left[1_{E^c}\cdot \left(\int_{t_0}^{(t_0+h)\wedge\tau_S}  \Lambda (s,X_s^{t_0,x_0,u_\epsilon } )ds-\psi(t_0+h,X_{t_0+h}^{t_0,x_0,u_\epsilon} )\right)\right]-\frac{ch^2}{a^4}. 	
\end{align*}
Hence, by the given assumptions, we have that: \begin{align}
&
-\psi (t_0,x_0)
\\&\begin{aligned}\geq \mathbb{E}&\Bigg[1_{E^c}\cdot\Bigg(\int_{t_0}^{t_0  +h\wedge \tau_S} (\delta -f (s, X_s^{t_0  , x_0  , u_\epsilon} ))ds  
\\&\begin{aligned}\hspace{-3 mm}+\left(\delta-G (\tau_S  \wedge\rho ,X_{\tau_S  \wedge\rho }^{t_0  , x_0  , u_\epsilon} )\right) \cdot 1_{\{\rho<t_0+h\}}-\psi(t_0+h,X_{t_0+h}^{t_0,x_0,u_\epsilon} ) \cdot 1_{\{t_0  +h=\rho \} }  \Bigg)\Bigg]&
\\-\frac{ch^2}{a^4}&
\end{aligned}  	
\end{aligned}  	
\nonumber\\&
\begin{aligned}
\geq &\mathbb{E}\Bigg[\delta\big(h+\mathbb{E}[ 1_{\{\rho<t_0+h\}}]\big)-\int_{t_0}^{t_0  +h\wedge \tau_S} f (s, X_s^{t_0  , x_0  , u_\epsilon} )ds
\\&\begin{aligned} -G (\tau_S  \wedge\rho ,X_{\tau_S  \wedge\rho }^{t_0  , x_0  , u_\epsilon} ) \cdot 1_{\{\tau_S  \wedge\rho< t_0  +h  \}}
 -\psi (t_0  +h, X_{t_0+h}^{t_0,x_0} ) \cdot 1_{\{t_0  +h=\rho  \}}  \Bigg]&
 \\-2 \frac{ch^2}{a^4} -\epsilon h.& 	\label{viscviscproofpsiweakdpp2}
\end{aligned}
\end{aligned}
\end{align}
Therefore combining (\ref{viscviscproofpsiweakdpp2}) and (\ref{viscviscproofpsiweakdpp}) and after rearranging we find that: \begin{equation}
\frac{2 ch^2}{a^4} +3\epsilon h\geq  \mathbb{E}\left[\delta\big(h+ \mathbb{E}[1_{\{\rho<t_0+h\}}\big)]\right]. \label{viscviscproofhpexpectationineq}	\end{equation}
Since $\mathbb{E}\left[h+1_{\{\rho<t_0+h\}}\right]\geq h$ using (\ref{viscviscproofhpexpectationineq}) we readily deduce that:
\begin{equation}
\frac{1}{2} \delta h\leq   \frac{ch^2}{a^4} +\frac{3}{2} \epsilon  h. 	
\end{equation} After which after dividing through by $h$  we find that: \begin{equation}
\frac{1}2 \delta - \left(\frac{c}{a^4}h +\frac{3}{2} \epsilon  \right)\leq 0. 	\label{viscviscproofhcontradictionineq}\end{equation}
We then deduce the result since both $h$ and $\epsilon$   can be made arbitrarily small which implies (\ref{viscviscproofhcontradictionineq}) yields a contradiction.

Next we prove that $V^-$ is a viscosity supersolution of (\ref{viscp1obstacle}). As in part (i), we prove the result by generating a contradiction, hence now suppose $\psi:[0,T]\times S\to \mathbb{R}$ is a test function with $\psi \in \mathcal{C}^{1,2} ([0,T ],S) $ and suppose $(t,x) \in [0,T]\times S$ is s.th. $\mathcal{M}V^{-}-\psi$ attains a local maximum at $(t,x)$.

In order to generate a contradiction, we assume that $\mathcal{M}V^- (s,x)-V^- (s,x)=\mathcal{M}V^- (s,x)-\psi(s,x)\leq 0$ and suppose that the supposition of the lemma is false so that $-\Lambda (s,x)-f (s,x )>0$, and consider constants $h,\delta >0$ s.th. $\forall  (s,x) \in [t_0,t_0+h]\times B_h (x)$ and $\mathcal{M}V^- (s,x)-\psi(s,x)\leq -\delta$  and $\Lambda (s,x)+f (s,x )\leq -\delta$. By Proposition \ref{Proposition 2.1.4.}, we can deduce the existence of an $\epsilon -$optimal strategy $\beta^\epsilon  \in \mathcal{B}_{(0,T)}$ to which the associated stopping time is $\beta^\epsilon  (u)\equiv \rho^\epsilon  \in \mathcal{T}$ for all $u \in \mathcal{U} $ (against $V^-)$ s.th. \begin{align}
&\begin{aligned}
\psi(t_0,x_0)&\leq \sup_{\rho \in \mathcal{T}} \mathbb{E}\Bigg[\int_{t_0}^{(t_0+h)\wedge\rho\wedge\tau_S} f (s,X_s^{t_0,x_0,\alpha(\rho)}  )ds+\sum_{j\geq  1} c(\tau_j,\xi_j ) \cdot 1_{\{\tau_j<(t_0+h)\wedge\rho\wedge\tau_S \}}  \\&+G(\rho\wedge\tau_S,X_{\rho\wedge\tau_S}^{t_0,x_0,\alpha(\rho)} ) \cdot 1_{\{\rho\wedge\tau_S\leq t_0+h\}}  +V^- (t_0+h,X_{t_0+h}^{t_0,x_0,\alpha(\rho)} ) \cdot 1_{\{\rho\wedge\tau_S> t_0+h\}}  \Bigg]\nonumber\end{aligned} 	
\\&
\begin{aligned}
\leq &\mathbb{E}\Bigg[\int_{t_0}^{(t_0+h)\wedge\rho^\epsilon \wedge\tau_S} f (s,X_s^{t_0,x_0,\alpha(\rho^\epsilon  )} )ds+\sum_{j\geq  1} c(\tau_j,\xi_j ) \cdot 1_{\{\tau_j<(t_0+h)\wedge\rho^\epsilon \wedge\tau_S\}}  \\&\begin{aligned}\quad+G(\rho^\epsilon\wedge\tau_S,X_{\rho^\epsilon\wedge\tau_S}^{t_0,x_0,\alpha(\rho^\epsilon)} ) \cdot 1_{\{\rho^\epsilon\wedge\tau_S\leq t_0+h\}}  +V^- (t_0+h,X_{t_0+h}^{t_0,x_0,\alpha(\rho^\epsilon  )} ) \cdot 1_{\{\rho^\epsilon\wedge\tau_S> t_0+h\}}  \Bigg]&
\\+\epsilon h .& 	\label{viscviscproofstep1P2}
\end{aligned}
\end{aligned}
\end{align}
After re-employing the estimate (\ref{viscviscproofphipsiineq}), we find that: \begin{align}\nonumber
\psi(t_0,x_0)\leq \mathbb{E}\Bigg[\int_{t_0}^{(t_0+h)\wedge\rho^\epsilon \wedge\tau_S} f (s,X_s^{t_0,x_0,u_0} )ds&+G(\rho^\epsilon\wedge\tau_S,X_{\rho^\epsilon\wedge\tau_S}^{t_0,x_0,u_0} ) \cdot 1_{\{\rho^\epsilon\wedge\tau_S\leq t_0+h\}}
\\&\begin{aligned}
+V^- (t_0+h,X_{t_0+h}^{t_0+h,u_0} ) \cdot 1_{\{\rho^\epsilon\wedge\tau_S> t_0+h\}}  \Bigg]&
\\+2\epsilon h. 	&
\label{viscviscproofstep1P2psipseduodpp}
\end{aligned}
\end{align}
Now by Remark \ref{remark 1.2.7.}, we have that $-\delta \geq \mathcal{M}V^- (s,x)-\psi(s,x)\geq V^- (s,x)-\psi(s,x)$; that is $\psi(s,x)\geq V^- (s,x)+\delta,\;  \forall  (s,x) \in [t_0,t_0+h]\times B_h (x)$. Using the definition of $\Lambda$   and the set $E$ introduced earlier and, again applying It\={o}'s formula, by similar reasoning as part (i), we find that: 
\begin{align}
\psi (t_0,x_0)&\geq \mathbb{E}\Bigg[\Big(\int_{t_0}^{(t_0+h)\wedge\tau_S} -\Lambda  (s,X_s^{t_0  , x_0  ,\alpha (\rho^\epsilon   )} )  ds+\psi(t_0  +h, X_{t_0+h}^{t_0, x_0  ,\alpha (\rho^\epsilon   )} )\Big) \cdot 1_{E^c} \Bigg]-\frac{ch^2}{a^4}
\nonumber\\
&\begin{aligned}
\geq \mathbb{E}\Bigg[\Big(\int_{t_0}^{(t_0+h)\wedge\tau_S} \left(\delta +f (s,X_s^{t_0  , x_0  ,\alpha (\rho^\epsilon   )} )\right)  ds + V^- (t_0  +h, X_{t_0+h}^{t_0, x_0  ,\alpha (\rho^\epsilon   )} )\Big) \cdot 1_{E^c} \Bigg]&
\\+\delta-\frac{ch^2}{a^4}.&
\end{aligned}
\end{align} 	
Employing similar reasoning as in part (i), and again re-employing the estimate (\ref{viscviscproofphipsiineq}) we find that: 
\begin{align}
&\begin{aligned}
\psi(t_0,x_0)&\geq\mathbb{E}\Bigg[\delta\big(h+ \mathbb{E}[1_{\{\rho^\epsilon<t_0+h\}}]\big)+\int_{t_0}^{(t_0+h)\wedge\tau_S} f (s,X_s^{t_0,x_0,\alpha(\rho^\epsilon) } )ds
\\&\begin{aligned}\;+G(\tau_S\wedge\rho,X_{\tau_S\wedge\rho^\epsilon}^{t_0,x_0,\alpha(\rho^\epsilon)} ) \cdot 1_{\{\tau_S\wedge\rho^\epsilon<t_0+h \}} -V^- (t_0  +h, X_{t_0+h}^{t_0, x_0  ,\alpha (\rho^\epsilon   )} ) \cdot 1_{\{t_0+h=\rho^\epsilon \}}  \Bigg]&
\\-2 \frac{ch^2}{a^4} -\epsilon h&
\end{aligned}
\end{aligned}  	
\nonumber\\
&
\begin{aligned}
\geq \mathbb{E}\Bigg[\delta \mathbb{E}[1_{(\rho^\epsilon<t_0+h)}]+\int_{t_0}^{(t_0+h)\wedge\tau_S}& f(s,X_s^{t_0,x_0,u_0} )ds+G(\tau_S\wedge\rho^\epsilon,X_{\tau_S\wedge\rho}^{t_0,x_0,u_0} ) \cdot 1_{\{\tau_S\wedge\rho^\epsilon<t_0+h \}} \\&-V^- (t_0  +h, X_{t_0+h}^{t_0, x_0  ,u_0} ) \cdot 1_{\{t_0+h=\rho^\epsilon \}}  \Bigg]-2 \frac{ch^2}{a^4} -2\epsilon h, 	\label{viscviscproofstep1P2penultim}\end{aligned}
\end{align}
where we have used the fact that $h>0$ which implies that $\delta\big(h+ \mathbb{E}[1_{\{(\rho^\epsilon<t_0+h)\}}]\big)>\delta\mathbb{E}[1_{\{(\rho^\epsilon<t_0+h)\}}]$.

Hence, combining (\ref{viscviscproofstep1P2penultim}) with (\ref{viscviscproofstep1P2}) and since:
\begin{equation}
4\epsilon h\geq \delta h-2 \frac{ch^2}{a^4} , 	\label{viscviscproofstep1P2hineq}
\end{equation}
for $h$ small enough $h<1$, we therefore find that:
\begin{equation}
\frac{1}{2}\delta - \left(2\epsilon +  \frac{ch}{a^4}\right)\leq 0, 	\end{equation}
which is a contradiction since both $\epsilon$  and $h$ can be made arbitrarily small --- hence we deduce the thesis.

\end{proof} 
Crucially, Lemma \ref{Lemma 2.6.2.} establishes the viscosity solution property of the game which, in conjunction with the DPP (Theorem \ref{Dynamic_programming_principle_for_stochastic_differential_games_of_control_and_stopping_with_Impulse_Controls}) is derived from first principles. We have therefore succeeded in characterising the value of the game in terms of a viscosity solution

\begin{theorem}\label{Theorem 2.6.5.} 
If the value of the game $V$ exists, then $V$ is a viscosity solution to the HJBI equation (\ref{viscp1obstacle}).
\end{theorem}
\begin{proof}

Let us firstly recall that by (\ref{viscviscproofimpulseremovalestimatepsi}) and selecting $h$ s.th. $h< \tau_S -t_0$ we have the following inequality: 
\begin{align}\nonumber
&\hspace{-20 mm}\psi (t_0,x_0)
\\&\hspace{-20 mm}\begin{aligned}\geq  \mathbb{E}\Bigg[\int_{t_0}^{\rho\wedge (t _0  +h)} f (s,X_s^{t_0  , x_0  , u_\epsilon} )ds  &+G (\rho\wedge\tau_S,X_\rho^{t_0  , x_0  , u_\epsilon} ) \cdot 1_{\{\rho\wedge\tau_S\leq  t_0  +h  \}}
\\&\begin{aligned}\;+\psi (t_0  +h, X_{t_0  +h}^{t_0  , x_0} ) \cdot 1_{\{t_0  +h<\rho\wedge\tau_S\}  }  \Bigg]&
\\-2\epsilon h&.
\end{aligned}
\end{aligned}
\end{align}
Moreover, since $V^+-\psi$ attains a local minimum at $(t,x)$, we can deduce the existence of a constant $\delta >0$ s.th. for $ (t,x) \in [0,T]\times S$: 
\begin{equation}
V^+ (t,x)-\psi(t,x)\geq 0 \hspace{2 mm} \text{ for   }\hspace{2 mm} |(t,x)-(t_0,x_0)|\leq \delta .
\end{equation} 
Additionally, by Lemma \hyperlink{Lemma A.1.3}{A.1.3}, we can deduce the existence of a constant $c>0$ s.th. $\forall  t \in [t_0,T[ $ we have that: 
\begin{equation}
\mathbb{E}|X_t^{t_0,x_0}-x_0 |\leq c|t-t_0 |^{\frac{1}{2}}.
\end{equation}
We can therefore deduce the existence of a sequence $t_n\downarrow t_0$ for which $X_{t_n}^{t_0,x_0}\to x_0$ as $n\to  \infty$. Let us now define the closed balls $\{B_n\}_{n\geq 1}$ by the following: \begin{gather*}
B_n:=\{|X_{t_m}^{t_0,x_0}-x_0 |\leq \delta \text{  } \forall  m\geq n\},\\
B_n\downarrow B\equiv \cup_{n\geq 1} B_n.
\end{gather*}
Further, let us now introduce the sequence of stopping times: \begin{equation} \tau_m=\sum_{n=1}^{ \infty} t_{n+m} \cdot 1_{\{B_n\backslash B_{(n-1)}\}}\wedge\rho-.
\end{equation} 
Hence, by (\ref{viscviscproofimpulseremovalestimatepsi}) we have that: \begin{equation}
  \psi (t_0,x_0)\geq \mathbb{E}\Bigg[\int_{t_0}^{t_m} f (s,X_s^{t_0  , x_0  , u_\epsilon} )ds  +\psi (t_m  , X_{t_m}^{t_0  , x_0} )\Bigg]-2\epsilon ( t_m  - t_0 ). 	\label{viscHJBIproofseq}
  \end{equation}
After applying It\={o}'s formula for c\`{a}dl\`{a}g semi-martingale processes to (\ref{viscHJBIproofseq}), we find that: \begin{equation}
0\geq \mathbb{E}\Bigg[\int_{t_0}^{t_m} \frac{\partial \psi}{\partial t} (s,X_s^{t_0,x_0,u_\epsilon } )+\mathcal{L}\psi(s,X_s^{t_0,x_0,u_\epsilon } )+f(s,X_s^{t_0,x_0,u_\epsilon } ))ds\Bigg]-2\epsilon ( t_m  - t_0 ).  	\label{viscHJBIproofprelim}\end{equation}
Then, after dividing both sides of (\ref{viscHJBIproofprelim}) by $( t_m  - t_0 ) $ and taking the limit $m\to  \infty$, we deduce that: \begin{equation}
0\geq \frac{\partial \psi}{\partial t} (t,x)+\mathcal{L}\psi(t,x)+f(t,x),\quad \forall (t,x)\in[0,T]\times S, 
\end{equation} 
which proves the subsolution property. We can prove the supersolution property analogously by firstly using (\ref{viscviscproofstep1P2psipseduodpp}) and applying similar steps after which the thesis is proved.
 \end{proof}

The following result establishes the equality of the two value functions $V^-$ and $V^+$; we defer the proof of the following result to the appendix:\

\begin{theorem}[Comparison Principle]\label{Theorem 2.6.5}
Let $V^-:[0,T]\times S\to \mathbb{R}$ be a continuous bounded viscosity subsolution to (\ref{viscp1obstacle}) and let $V^+:[0,T]\times S\to \mathbb{R}$ be a continuous bounded viscosity supersolution to (\ref{viscp1obstacle}). Also suppose that for all $t \in [0,T]$ we have that $V^- (\cdot,X_{T}^{t,\cdot} )\leq V^+ (\cdot,X_{T}^{t,\cdot} )$ then we have that:
\begin{equation}
V^- (t,x)\leq V^+(t,x), \qquad \forall  (t,x) \in [0,T]\times S.
\end{equation}
\end{theorem}
\begin{corollary}[The Game Admits a Value]\label{Corollary 2.6.6}
\end{corollary}
To prove Theorem \ref{viscvaluefunctionequalitytheorem} it remains only to reverse the inequality (\ref{viscvaluefunctioninequalitystandard}) therefore proving that $V^- (\cdot,X_{T}^{t,\cdot} )\leq V^+ (\cdot,X_{T}^{t,\cdot} )$ --- a result that follows directly from the comparison principle for the game. Indeed, Theorem \ref{viscvaluefunctionequalitytheorem} and Corollary \ref{Corollary 2.6.6} then follow as direct consequences to the viscosity solutions results of Lemma \ref{Lemma 2.6.2.} in conjunction with the comparison principle.

\section{Appendix} \hypertarget{VI}{}\

\noindent \textbf{Lemma A.1.}\hypertarget{Lemma A.1}{} Let $\hat{\mu} \in [0,T]$ be some $\epsilon -$optimal strategy against $V^{±} (t,x) $ for any $ (t,x) \in [0,T]\times  S$ then there exists some $\eta >0$ s.th. the strategy $\hat{\mu}$ remains 2$\epsilon -$optimal against $V^± (t,y^j ) $ for any $y^j \in B(x,\eta ) $.

\begin{proof}[Proof of Lemma A.1]

We do the proof for the function $V^-$ with the proof for $V^+$ being analogous. Denote by $\rho'\equiv \hat{\mu}(u) $ where $\rho' \in [0,T] $. Since the strategy $\hat{\mu}$ is $\epsilon -$optimal against $V^- (t,x) $ we have that for some $\epsilon >0$:  \begin{equation}
V^- (t,x)\leq  \inf_{u \in \mathcal{U}}J^{(u,\rho' )} (t,x)+\epsilon . 	 
\end{equation}
\indent Now by Proposition \ref{Proposition 2.1.4.} we can deduce the existence of some constants $c_1,c_2>0$ s.th. for any $y^j \in B(x,\eta ),u \in \mathcal{U},\rho \in \mathcal{T}$:
\begin{equation}
\left|J^{(u,\rho)} (t,y^j )-J^{(u,\rho)} (t,x)\right|\leq c_1 \left|x-y^j \right|, 	\end{equation}
hence, 
\begin{align*}
\inf_{u \in \mathcal{U}}&J^{(u,\rho' )} (t,y^j )\geq  \inf_{u \in \mathcal{U}}J^{(u,\rho' )} (t,x)-c_1 \left|x-y^j \right| 	
\\&\geq V^- (t,x)-c_1 \left|x-y^j\right|-\epsilon  	
\\&\geq  V^-  (t, y^j  )-\left( c_1  + c_2 \right)\left|x- y^j \right|-\epsilon   	
\\&\geq  V^-  (t, y^j  )-2\epsilon  , 	\end{align*}
where the last line follows provided that $|x-y^j |\leq \eta :=\epsilon (c_1+c_2 )^{-1}$, from which we deduce the thesis after applying $\sup_{u \in \mathcal{U}}$ to both sides and since $\epsilon$  is arbitrary.
\end{proof}
The following results relate the dependence of $X^{t,x,\cdot}$ in the initial point $(t,x) \in [0,T]\times S$. 

\noindent\textbf{Lemma A.2.}\ \hypertarget{Ap.3}{} For all $ (t,x' ),(t,x) \in [0,T]\times  S$ and for any $s \in [0,T] $, we can deduce the existence of a constant $c>0$ s.th.: 
\begin{equation}
\mathbb{E}\left|X_s^{t,x,\cdot}  - X_s^{t,x',\cdot}  \right|\leq c \left|x'-x\right|. 	\end{equation}
\begin{proof}[Proof of Lemma A.2]
Using It\={o}'s lemma and It\={o} isometry, we readily observe that for all $(t,x),(t,x')\in [0,T]\times S$: 

\begin{align*} 
&\qquad\qquad\qquad\qquad\qquad\qquad\mathbb{E}\left| X_s^{t,x,\cdot}-X_s^{t,x',\cdot}\right|^2\nonumber\nonumber
\\&\nonumber
\begin{aligned}\nonumber
\leq | x-x' |^2&+\mathbb{E}\Bigg[2\int_{0}^s \left\langle X_r^{t',x,\cdot}-X_r^{t,x,\cdot},\mu(r,X_r^{t',x,\cdot} )-\mu(r,X_r^{t,x,\cdot} ) \right\rangle dr
\\&\begin{aligned}\;+2\int_{0}^s  \langle X_r^{t,x,\cdot}-X_r^{t,x',\cdot},\sigma(X_r^{t,x,\cdot} )-\sigma(r,X_r^{t,x',\cdot} ) \rangle  dW_r&
\\
+\left|\sigma(r,X_r^{t,x,\cdot} )-\sigma(r,X_r^{t,x',\cdot} )dW_r\right|^2 \Bigg]%
\end{aligned}
\end{aligned}
\\&\begin{aligned} 
\leq | x-x' |^2&+2\mathbb{E}\left[\int_{{0}}^s \left\langle X_r^{t,x,\cdot}-X_r^{t,x',\cdot},\mu(r,X_r^{t,x,\cdot} )-\mu(r,X_r^{t,x',\cdot} ) \right\rangle dr\right]
  \\&+\mathbb{E}\left[\int_{{0}}^s  \left|\sigma(r,X_r^{t,x,\cdot} )-\sigma(r,X_r^{t,x',\cdot} )\right|^2 dr\right], \label{a.1proofs1}
\end{aligned}\end{align*}
where we have used the standard properties of Brownian motion and the boundedness of $\sigma$ to eliminate the linear stochastic integral term.

By assumption \hyperlink{A.1.1.}{A.1.1.} (Lipschitzianity of $\sigma$ and $\mu$) we now observe that: 
\begin{align*} 
 &\qquad\qquad\qquad\qquad\qquad\qquad\mathbb{E}\left| X_s^{t,x,\cdot}-X_s^{t,x',\cdot}\right|^{2}
\\&\begin{aligned}
\leq |x-x'|^2+ \mathbb{E}\Bigg[2\int_{0}^s \left\langle X_r^{t,x,\cdot}-X_r^{t,x',\cdot},\mu( r,X_r^{t,x,\cdot} )-\mu(r,X_r^{t,x',\cdot} ) \right\rangle dr
\\&\hspace{-24 mm} +\left|\sigma( r,X_r^{t,x,\cdot} )-\sigma(r,X_r^{t,x',\cdot} )\right|^2 dr\Bigg]
\end{aligned} 	
\\&
\leq | x-x' |^2+\left(\text{Lip}(\sigma)+\text{Lip}(\mu)\right) \int_{0}^s  \mathbb{E}\left[\left|X_r^{t,x,\cdot}-X_r^{t,x',\cdot}\right|^2\right] dr 
\leq C' \left|x-x' \right|^2, 	
\end{align*}
where $C':= 1+T\| \mu_{\infty}\|[$Lip$(\sigma)+\text{Lip}(\mu)]$ and where the last line follows from Gronwall's lemma, after which we can readily deduce the result.\
\end{proof}
In a similar way to Lemma \hyperlink{Ap.3}{A.2.}, we can deduce the following estimates:\\
\textbf{Lemma A.3.} \hypertarget{Ap.4}{}
\begin{align*}&\mathbb{E}\left[\sup_{s\in [0, T] } \left|X_s^{t,x,\cdot}\right |^p   \right]\leq c\left(1+|x|^p\right).
\\&\mathbb{E}\left[\sup_{s\in [t,h] } \left|X_s^{t,x,\cdot}-x\right|^p   \right]\leq ch^{p/2} \left(1+|x|^p \right)
\\&\mathbb{E}\left[\sup_{s\in [t,t' ] } \left|X_s^{t',x',\cdot}-X_s^{t,x,\cdot} \right|^p   \right]\leq c\left|x'-x\right|^p+\left|t'-t\right|^{p/2} (1+|x|^p )|,
\end{align*}
\\
\begin{proof}
Denote by $x:= X_t^{t,x}$ and $x':= X_t^{t,x'}$, using It\={o}'s lemma we readily observe that for all $(t,x),(t,x')\in [0, T]\times S$ there exists some constants $c',c>0$ s.th. \begin{align}
\mathbb{E}&\left| X_s^{t,x,\cdot} \right|^2 \leq \left|x\right|^2+\mathbb{E}\left[\int_{t}^s  \left|\mu(r,X_r^{t,x,\cdot} )\right|^2 dr \right]+\mathbb{E}\left[\int_{t}^s  |\sigma(r,X_r^{t,x,\cdot} )|^2 dr \right]\nonumber
\\&\nonumber\leq \left|x\right|^2+c\mathbb{E}\left[\int_{t}^s  |X_r^{t,x,\cdot} |^2 dr \right]
\\&\nonumber
\leq |x|^2+c\mathbb{E}\left[\int_{t}^s (|X_r^{t,x,\cdot} |^2+\mathbb{E}[|X_r^{t,x,\cdot}-X_t^{t,x,\cdot} |^2 ])dr\right]
\\&
\leq c(1+|x|^2 )+c'\mathbb{E}\left[\int_{t}^s (1+\sup_{r\in [0,T]} \mathbb{E}[|X_r^{t,x,\cdot} |^2  )dr\right],	\label{lemmaa.2proofs1}\end{align}
by assumptions  \hyperlink{AA.2.}{A.2.},  \hyperlink{A.1.1.}{A.1.1.} and Fubini's theorem (and the smoothing theorem). Hence, after applying Gronwall's lemma to (\ref{lemmaa.2proofs1}), we immediately deduce the existence of some real-valued constant $c>0$ s.th:
\begin{equation} 
\mathbb{E}[|X_s^{t,x,\cdot}|^2]\leq c (1+|x|^2).	
\end{equation}
The required result is thereafter easily deduced.

The proofs of (ii) and (iii) proceed in a very much a similar way to the proof of Lemma \hyperlink{Ap.3}{A.2.} and are omitted. For (iii), we refer the reader to [Theorem 2.4, \cite{touz}] for the complete details.
\end{proof}

The proof of the comparison principle is an adaptation of the standard comparison theorem result, indeed, we prove Theorem \ref{Theorem 2.6.5.} by making the necessary adjustments to existing comparison theorem results as given in for example, \cite{cosso2013stochastic}.\

To prove Theorem \ref{Theorem 2.6.5.}, we first require the following definition and result:

\begin{definition}\label{Definition A.1.2.}
Let $\psi \in \mathcal{C}([0,T];\mathbb{R}^p ) $ be a lower semicontinuous function, then the parabolic subjet of $\psi$ at the point $ (t,x) \in [0,T]\times S$ which we denote by $J^{(2,-)} \psi(t,x) $ is the set of triples $ (M,r,q) \in  \mathbb{S}(p)\times \mathbb{R}\times S$ s.th. \begin{equation}
\psi(s,y)\geq \psi(s,x)+r(s-t)+\langle q,y-x\rangle+\frac{1}{2} \langle M(y-x),y-x \rangle+\mathcal{O}(|s-t|+|y-x|^2 )
)\end{equation}
as $s\to t $ or as $s\downarrow t$ when $t=0$ and $y\to x$. We can analogously define the parabolic superjet of $\psi$ at the point $(t,x) \in [0,T]\times S$ which we denote by $J^{(2,+)} \psi(t,x) $ by the following: \begin{equation}
J^{(2,+)} \psi(t,x)\equiv -J^{(2,-)} (-\psi)(t,x). \end{equation}

\end{definition}
Let us also introduce the following notation the convenience of which is immediate: suppose $\Lambda :\mathbb{S}(p)\times \mathbb{R}^p\times \mathcal{C}([0,T];\mathbb{R}^p )\times [0,T]\times \mathbb{R}^p\to \mathbb{R}$ then we define $\Lambda$  by: 
\begin{multline}
     \Lambda (M,r,\psi,m,q):=m-\sum_{i=1}^{p} \mu_i  (m,q) r_i+\frac{1}{2} \sum_{i,j=1}^{p} (\sigma\cdot\sigma^T )_{ij} (q) M_{ij} \\\quad+\sum_{j=1}^l  \int_{\mathbb{R}^p} {\psi(m,q+\gamma^{(j)} ) (m,q,z_j ))-\psi(m,q)-r\cdot\gamma^{(j)} ) (m,q,z_j )} \nu_j (dz_j ) +f(m,q). 
\end{multline}
 We note that using Definition \ref{Definition A.1.2.} we can obtain the following result --- the proof of which is standard and therefore omitted:

\begin{lemma}\label{Lemma A.1.3.1}
Let $\psi \in \mathcal{C}([0,T];\mathbb{R}^p ) $ be a lower (resp., upper) semicontinuous function, then $\psi$ is a viscosity supersolution (resp., subsolution) to the HJBI equation (\ref{viscp1obstacle}) iff:$\forall  (t,x) \in [0,T]\times S$ and $\forall  (M,m,r) \in  \bar{J}^{(2,-)} \psi(t,x)$ (resp., $\bar{J}^{(2,+)} \psi(t,x)$) we have that: \begin{equation}
\max\left\{\min\left[-\Lambda (M,r,\psi,m,x),V-G \right], V-\mathcal{M}V\right\}\geq 0  \hspace{1 mm}(\text{resp.,}\leq 0), 
\end{equation}
and,
\begin{equation}
\max\left\{V(\tau_S,x)-G(\tau_S,x),V(\tau_S,x)-\mathcal{M}V(\tau_S,x)\right\}\geq 0  \hspace{1 mm}(\text{resp.,}\leq 0);\qquad \forall x \in S. 
\end{equation}
\end{lemma}
Having stated the above results, we can now prove Theorem \ref{Theorem 2.6.5.}:
\begin{proof}[Proof of Theorem \ref{Theorem 2.6.5.}]

We prove the comparison principle using the standard technique as introduced in \cite{crandall1992user} --- namely we prove the result by contradiction. 
Suppose that the functions $V$ and $U$ are a viscosity subsolution and supersolution (respectively) to the HJBI equation (\ref{viscp1obstacle}), then to prove the theorem we must prove that under assumptions \hyperlink{A.1.1}{A.1.1} - \hyperlink{A.4}{A.4} we have that $V\leq U$ on $[0,T]\times S$.

Hence, let us firstly assume that $\forall  x \in S$: \begin{equation}
V(T,x)\leq U(T,x).	\end{equation}
Moreover, let us also assume that: \begin{equation}
M:=\sup_{[0,T]\times S}(V-U)>0.\label{viscviscappendproofMVUineq}
\end{equation}
Now by Proposition \ref{Proposition 2.1.5} we know that $V$ and $U$ are bounded hence for some $\epsilon ,\alpha,\eta >0$ s.th. $\forall  (t,s,x,y) \in [0,T]^2\times S^2$: \begin{equation}
M_{\epsilon ,\alpha,\eta}:=\hspace{-2 mm}\max_{(t,s,x,y) \in [0,T]^2\times S^2} \hspace{-4 mm}V(t,x)-U(s,y)-\frac{(|t-s|^2+|x-y|^2)}{2\epsilon} -\frac{\alpha}{2} (|x|^2+|y|^2 )+\eta t),\end{equation}
is both a finitely bounded quantity and has some maximum which is achieved by a point (which depends on ($\epsilon ,\alpha,\eta$ )) which we shall denote by ($\bar{t},\bar{s},\bar{x},\bar{y} ) \in [0,T]^2\times S^2$.
Now since there exist values $ (s,y) \in [0,T]\times S$ s.th. $M=M_{(\epsilon ,\alpha,\eta )} $, we have that:  \begin{equation}
0<M\leq  M_{\epsilon ,\alpha,\eta}   =V (\bar{t},̅\bar{x} )-U (\bar{s},\bar{y})-  \frac{(|\bar{t}-\bar{s} |^2  + |\bar{x}-\bar{y}|^2 }{2\epsilon}   -  \frac{\alpha}{2}  (|\bar{x} |^2  + |\bar{y}|^2  )+\eta \bar{t} ).\end{equation}
Hence, \begin{equation}
\lim_{\epsilon \downarrow 0}\frac{(|\bar{t}-\bar{s} |^2  + |\bar{x}-\bar{y}|^2 }{2\epsilon} <V (\bar{t},̅\bar{x} )-U (\bar{s},\bar{y})-  \frac{\alpha}{2}  (|\bar{x} |^2  + |\bar{y}|^2  )+\eta \bar{t} ).
\end{equation}

Now, since the RHS is composed of finitely bounded terms and the LHS is non-negative, we readily conclude that $\lim_{\epsilon \downarrow 0}\frac{(|\bar{t}-\bar{s} |^2  + |\bar{x}-\bar{y}|^2 )}{2\epsilon}  =0$ and hence we observe that $|\bar{t}-\bar{s} |^2  + |\bar{x}-\bar{y}|^2 \to 0$ as $\epsilon \downarrow 0$.

Moreover, if we denote by $ (s_n,y_n ),(t_n,x_n ) \in [0,T]\times S$ and $\epsilon_n>0$ a triple of bounded sequences s.th. $ (s_n,y_n ),(t_n,x_n )\to (t,x) $ as $\epsilon_n\to 0$ we have that: 
\begin{align*}
&M  \leq \lim_{\epsilon ,\alpha,\eta  \downarrow 0} M_{\epsilon ,\alpha,\eta}
\\&\leq \lim_{\epsilon ,\alpha,\eta  \downarrow 0}\sup_{[0,T]\times S}[V (t_n,x_n )-U(s_n,y_n)-\frac{(|t-s|^2+|x-y|^2)}{2\epsilon} -\frac{\alpha}{2} (|x|^2+|y|^2 )+\eta t]
\\&=\sup_{[0,T]\times S}\lim_{\epsilon ,\alpha,\eta  \downarrow 0}[V (t_n,x_n )-U (s_n,y_n )
-\frac{(|t-s|^2+|x-y|^2)}{2\epsilon}-\frac{\alpha}{2} (|x|^2+|y|^2 )+\eta t]
\\&=\sup_{[0,T]\times S}[V (t,x)-U (t,x)]=M. 
\end{align*} 
We therefore readily deduce that: \begin{equation}
\lim_{(\epsilon ,\alpha,\eta  \downarrow 0)} M_{\epsilon ,\alpha,\eta}   =M . \end{equation}
 We now invoke Ishii's lemma (e.g. as in \cite{crandall1992user}) to the sequence $\{(t_n,x_n,y_n )\}_n$ so that we deduce the existence of a pair of triples $ (p_V^n,q_V^n,M_n ) \in \bar{J}^{(2,+)} V(t_n,x_n) $ and $ (p_U^n,q_U^n,N_n ) \in \bar{J}^{(2,-)} U(t_n,y_n) $ s.th the following statements hold: \begin{align}
p_V^n-p_U^n&=\partial_s \psi_n (t_n,x_n,y_n )=2(t_n-t_0 ) ,\\
q_V^n&=D_x \psi_n (t_n,x_n,y_n ),\\ q_U^n&=D_y \psi_n (t_n,x_n,y_n )	 ),
\end{align}
and
\[ 
\left (
  \begin{tabular}{cc}
  $M_n$ & $0$ \\
  $0$ & $-N_n$
  \end{tabular}\right )
\leq A_n +\frac{1}{2}A_n^2, \]
where $A_n:=D_{xy}^2 \psi_n (t_n,x_n,y_n ) $. Now we note that by the viscosity subsolution property of $V$ we have that: \begin{equation}
V(t_n,x_n )-p_V^n-\langle \mu(t_n,x_n ),q_V^n \rangle-\frac{1}{2} tr(\sigma\cdot\sigma' (t_n,x_n ) M_n )-f(t_n,x_n )\leq 0, 	\label{viscviscappendISHIIV} \end{equation}
And similarly, by the viscosity supersolution property of $U$ we have that:
\begin{equation}
U(t_n,y_n )-p_V^n-\langle\mu(t_n,y_n ),q_V^n \rangle-\frac{1}{2}tr(\sigma\cdot\sigma' (t_n,y_n ) N_n )-f(t_n,y_n )\geq 0.	\label{viscviscappendISHIIU} \end{equation}
Now subtracting (\ref{viscviscappendISHIIU}) from (\ref{viscviscappendISHIIV}) yields the following: \begin{gather}\nonumber
V(t_n,x_n )-U(t_n,y_n )\\\begin{split}\leq p_V^n-p_V^n&+\langle\mu(t_n,x_n ),q_V^n \rangle-\langle\mu(t_n,y_n ),q_V^n \rangle+\frac{1}{2} tr(\sigma\cdot\sigma' (t_n,x_n ) M_n )\\&-\frac{1}{2} {\rm tr}(\sigma\cdot\sigma' (t_n,y_n ) N_n )+f(t_n,x_n )-f(t_n,y_n )\leq 0.	\label{viscviscappendISHIIVminusU} \end{split}\end{gather}
We now use the fact that $ (s_n,y_n ),(t_n,x_n )\to (t,x) $ from which we now observe the following limits as $n\to  \infty$:\begin{equation}
\lim_{n\to  \infty}[p_V^n-p_V^n] =\lim_{n\to  \infty} [t_n-t_0]=0,	\label{viscviscappendPTlim} \end{equation}
and for some $c>0$: \begin{equation}
\lim_{n\to  \infty}\langle \mu(t_n,x_n ),q_V^n \rangle-\langle \mu(t_n,y_n ),q_V^n \rangle \leq c \lim_{n\to  \infty} |x_n-y_n |=0,	 \end{equation}
using the Lipschitzianity of $\mu$. Lastly we observe, using  that:
\begin{equation}\frac{1}{2}  \lim_{n\to  \infty} [{\rm tr}(\sigma\cdot\sigma' (t_n,x_n ) M_n )-{\rm tr}(\sigma\cdot\sigma' (t_n,y_n ) N_n )]=0.	\label{viscviscappendMNlim} \end{equation}
Hence, putting (\ref{viscviscappendPTlim}) - (\ref{viscviscappendMNlim}) together with (\ref{viscviscappendISHIIVminusU}) yields a contradiction since we observe that:\begin{equation}
\lim_{n\to  \infty} [V(t_n,x_n )-U(t_n,y_n )]\leq 0, 	\end{equation}
which is a contradiction to (\ref{viscviscappendproofMVUineq}).
\end{proof}
\section{Conclusion}
Using standard assumptions, we proved that stochastic differential games involving an impulse controller and a stopper with diffusion system dynamics admit a value. We also proved that the value of the game is a viscosity solution to a HJBI equation which is represented by a double obstacle quasi-integro-variational inequality. \

The game studied investigated in this paper is one in which the payoff structure is zero-sum; an interesting question which arises naturally is the existence and characterisation of stable equilibria for the non-zero-sum equivalent of the game.\

Expectedly, arguments similar to that given in \cite{buckdahn2004nash} (wherein the Nash equilibrium payoffs and controls for a stochastic differential game with continuous controls were characterised) can be invoked without laborious adaptation.\

% An interesting question for future research is investigating the above controller-stopper framework (in which the controller uses impulse controls to modify the state process) when either or both of the players only has access to partial information. Of particular interest is partial state information- that is, a system in which the state process is adapted to some subset of the canonical filtration.

\bibliographystyle{siamplain}

\end{document}

%% file: ex_shared.tex
\usepackage{lipsum}
\usepackage{amsfonts}
\usepackage{graphicx}
\usepackage{epstopdf}
\usepackage{algorithmic}
\ifpdf
  \DeclareGraphicsExtensions{.eps,.pdf,.png,.jpg}
\else
  \DeclareGraphicsExtensions{.eps}
\fi

\newsiamremark{remark}{Remark}
\newsiamremark{hypothesis}{Hypothesis}
\crefname{hypothesis}{Hypothesis}{Hypotheses}
\newsiamthm{claim}{Claim}

\headers{Control and Stopping Involving Impulsive Control}{David Mguni}

\title{A Viscosity Approach to Stochastic Differential Games of Control and Stopping Involving Impulsive Control}

\author{David Mguni\thanks{ 
  (\email{davidmguni@hotmail.com}).}{ Quantitative and Applied Spatial Economic Research Laboratory, University College London, Gower Street, London, WC1E 6BT, UK.\newline Centre for Doctoral Training in Financial Computing \& Analytics, University College London, Gower Street, London, WC1E 6BT, UK. }}
\usepackage{amsopn}

\makeatletter
\newcommand*{\addFileDependency}[1]{% argument=file name and extension
  \typeout{(#1)}% latexmk will find this if $recorder=0 (however, in that case, it will ignore #1 if it is a .aux or .pdf file etc and it exists! if it doesn't exist, it will appear in the list of dependents regardless)
  \@addtofilelist{#1}% if you want it to appear in \listfiles, not really necessary and latexmk doesn't use this
  \IfFileExists{#1}{}{\typeout{No file #1.}}% latexmk will find this message if #1 doesn't exist (yet)
}
\makeatother

\newcommand*{\myexternaldocument}[1]{%
    \externaldocument{#1}%
    \addFileDependency{#1.tex}%
    \addFileDependency{#1.aux}%
}